\documentclass[12pt]{article}   
\usepackage[utf8]{inputenc}
\usepackage{amsmath,amssymb,amsthm,geometry, wasysym}
\usepackage{graphicx, verbatim, caption, subcaption}
\usepackage{xcolor}
\usepackage{fullpage}
\usepackage[colorlinks=true,linkcolor=blue]{hyperref}
\hypersetup{
   citecolor=blue,
   linkcolor=blue,
}
\usepackage[psamsfonts]{eucal}
\usepackage{microtype}
\usepackage{mathtools}

\title{Sums of GUE matrices and concentration of hives from correlation decay of eigengaps}
\author{Hariharan Narayanan\thanks{School of Technology and Computer Science, TIFR Mumbai} \and Scott Sheffield\thanks{Department of Mathematics,  MIT} \and Terence Tao\thanks{Department of Mathematics,  UCLA}}
\date{\today}
\begin{document}
\newtheorem{theorem}{Theorem}
\newtheorem{lemma}[theorem]{Lemma}
\newtheorem{conjecture}[theorem]{Conjecture}
\newtheorem{question}[theorem]{Question}
\newcommand{\Prob}{\ensuremath{\text{Pr}}}
\newcommand{\E}{\ensuremath{\mathbb E}}
\newcommand{\R}{\ensuremath{\mathbb R}}
\newcommand{\cX}{\ensuremath{\mathcal X}}
\newcommand{\cT}{\ensuremath{\mathcal T}}
\newcommand{\K}{\ensuremath{\mathcal K}}
\newcommand{\F}{\ensuremath{\psi}}
\newcommand{\tr}{\ensuremath{{\scriptscriptstyle\mathsf{T}}}}
\newcommand{\inner}[1]{\left\langle #1 \right\rangle}

\newcommand{\tw}{\tilde{w}}
\newcommand{\lab}{\label} \newcommand{\M}{\ensuremath{\mathcal M}} \newcommand{\ra}{\ensuremath{\rightarrow}} \def\eee{{\mathrm e}} \def\a{{\mathbf{\alpha}}} \def\de{{\mathbf{\delta}}} \def\De{{{\Delta}}}  \def\m{{\mathbf{\mu}}}
\def\tm{{\tilde{\mu}}} \def\var{{\mathrm{var}}} \def\beq{\begin{eqnarray}} \def\eeq{\end{eqnarray}} \def\ben{\begin{enumerate}}
\def\een{\end{enumerate}} \def\bit{\begin{itemize}}
\def\bel{\begin{lemma}}
\def\eel{\end{lemma}}
\def\eit{\end{itemize}} \def\beqs{\begin{eqnarray*}} \def\eeqs{\end{eqnarray*}} \def\bel{\begin{lemma}} \def\eel{\end{lemma}}
\newcommand{\N}{\mathbb{N}} \newcommand{\Z}{\mathbb{Z}} \newcommand{\Q}{\mathbb{Q}} \newcommand{\C}{\mathcal{C}} \newcommand{\CC}{\mathcal{C}}
\newcommand{\T}{\mathbb{T}} \newcommand{\A}{\mathbb{A}} \newcommand{\x}{\mathbf{x}} \newcommand{\y}{\mathbf{y}} \newcommand{\z}{\mathbf{z}}  \renewcommand{\b}{\mathbf{b}} 
\newcommand{\tb}{\tilde{\mathbf{b}}}
\newcommand{\n}{\mathbf{n}} \newcommand{\I}{I} \newcommand{\II}{\mathcal{I}} \newcommand{\EE}{\mathbb{E}} \newcommand{\p}{\mathbb{P}}
\newcommand{\PP}{\mathcal P} \newcommand{\BB}{\mathcal B} \newcommand{\HH}{\mathcal H} \newcommand{\e}{\mathbf{e}} \newcommand{\one}{\mathrm{1}}
\newcommand{\LL}{\Delta} \newcommand{\MM}{\mathcal M}\newcommand{\NN}{\mathcal N} \newcommand{\la}{\lambda} \newcommand{\Tr}{\text{Trace}} \newcommand{\aaa}{\alpha}
\def\eee{{\mathrm e}}  \def\eps{{\epsilon}} \def\A{{\mathcal{A}}} \def\ie{i.\,e.\,} \def\g{G}
\def\vol{\mathrm{vol}}\newcommand{\tc}{{\tilde{c}}}
\newcommand{\tP}{\tilde{P}}
\newcommand{\sig}{\sigma}
\newcommand{\wn}{w^{(n)}}
\newcommand{\wone}{w^{(n_1)}}
\newcommand{\rr}{\mathbf{r}}
\newcommand{\tf}{\hat{F}}
\newcommand{\tG}{\hat{G}}
\newcommand{\tg}{\hat{G}}
\newcommand{\tv}{\hat{V}}
\newcommand{\tk}{\hat{K}}
\newcommand{\rmm}{\rho_{min}}
\newcommand{\La}{\Lambda}
\newcommand{\sy}{\mathcal{S}_y}
\newcommand{\cc}{\mathbf{c}}
\newcommand{\RR}{\mathbb{R}}
\newcommand{\dist}{\mathbf{d}}
\newcommand{\G}{\mathcal{G}} \newcommand{\fat}{\mathrm{fat}}
\newcommand{\cuf}{\textsf{CurvFit}}
\newcommand{\hf}{{\mathcal{H}}_{f}}
\newcommand{\hb}{{\mathcal{H}}_{b}}
\newcommand{\s}{\mathcal{S}}
\newcommand{\us}{\underline{s}}
\renewcommand{\u}{u}
\renewcommand{\v}{v}
\newcommand{\J}{\mathcal{J}}
\renewcommand{\sf}{\mathcal{sf}}
\renewcommand{\g}{\mathcal{G}}
\renewcommand{\L}{\mathbb{L}} 
\renewcommand{\i}{k}
\renewcommand{\j}{\ell}
\newcommand{\ubn}{\hat{u}^{(n_1)}_b}
\renewcommand{\d}{\partial}
\newcommand{\dperp}{\partial_\perp}
\newcommand{\curl}{A}
\renewcommand{\div}{\nabla_\bullet}
\renewcommand{\L}{\mathbb L}
\renewcommand{\C}{\mathcal C}
\newcommand{\f}{\mathbf{f}}
\newcommand{\ta}{\tilde{a}}

\newcommand{\tA}{\tilde{A}}
\newcommand{\ent}{\mathrm{ent}}
\newcommand{\hess}{\nabla^2}

\theoremstyle{plain}

\newtheorem{thm}{Theorem}
\newtheorem{corollary}{Corollary}
\newtheorem{proposition}{Proposition}
\newtheorem{defn}{Definition}
\newtheorem{definition}{Definition}
\newtheorem{example}{Example}
\newtheorem{remark}{Remark}
\newtheorem{claim}{Claim}
\newtheorem{obs}{Observation}
\numberwithin{equation}{section}
\numberwithin{figure}{section}
\renewcommand{\a}{\alpha}
\renewcommand{\b}{\beta}
\renewcommand{\g}{\gamma}
\newcommand{\Ia}{I_\alpha}
\newcommand{\Ib}{I_\beta}
\newcommand{\Ig}{I_\gamma}
\newcommand{\spec}{\mathtt{spec}}
\newcommand{\Spec}{\mathtt{Spec}}
\newcommand{\Leb}{\mathrm{Leb}}
\newcommand{\diag}{\mathrm{diag}}
\renewcommand{\ta}{\tilde{\alpha}}
\renewcommand{\tb}{\tilde{\beta}}
\renewcommand{\tg}{\tilde{\gamma}}
\newcommand{\tla}{\tilde{\lambda}}
\newcommand{\tmu}{\tilde{\mu}}
\newcommand{\tnu}{\tilde{\nu}}
\newcommand{\tilh}{\tilde{h}}
\newcommand{\stp}{\mathrm{stop}}
\newcommand{\HIVE}{\mathtt{HIVE}}
\newcommand{\AHIVE}{\mathtt{AUGHIVE}}
\newcommand{\HORN}{\mathtt{HORN}}
\newcommand{\oct}{\mathbf{oct}}
\newcommand{\GT}{\mathtt{GT}}
\newcommand{\GL}{\mathrm{GL}}
\newcommand{\Var}{\operatorname{var}}
\newcommand{\cov}{\operatorname{cov}}
\newcommand{\CZ}{\mathrm{CZ}}
\newcommand{\lf}{\mathrm{lf}}
\newcommand{\weight}{\mathbf{wt}}
\newcommand{\edge}{{\diamond}}
 \renewcommand{\Tr}{\mathbf{tr}}
\renewcommand{\F}{\mathcal F}
\renewcommand{\RR}{\mathcal R}
\newcommand{\tX}{\tilde{X}}
\newcommand{\tY}{\tilde{Y}}
\newcommand{\D}{\mathcal{D}}
\newcommand{\bb}{\mathbf b}
\newcommand{\tbbb}{\mathbf \tilde{b}}
\newcommand{\wred}{w_0^{\mathrm{red}}}
\newcommand{\wblue}{w_0^{\mathrm{blue}}}
\newcommand{\wgreen}{w_0^{\mathrm{green}}}
\renewcommand{\P}{\mathbb{P}}
\newcommand{\rel}{\to}

\maketitle
\begin{abstract}  Associated to two given sequences of eigenvalues $\lambda_1 \geq \dots \geq \lambda_n$ and $\mu_1 \geq \dots \geq \mu_n$ is a natural polytope, the polytope of \emph{augmented hives} with the specified boundary data, which is associated to sums of random Hermitian matrices with these eigenvalues.  As a first step towards the asymptotic analysis of random hives, we show that if the eigenvalues are drawn from the GUE ensemble, then the associated augmented hives exhibit concentration as $n \ra \infty$.  Our main ingredients include a representation due to Speyer of augmented hives involving a supremum of linear functions applied to a product of Gelfand--Tsetlin polytopes; known results by Klartag on the KLS conjecture in order to handle the aforementioned supremum; covariance bounds of Cipolloni--Erd\H{o}s--Schr\"oder of eigenvalue gaps of GUE; and the use of the theory of determinantal processes to analyze the GUE minor process.
\end{abstract}
\newpage
\tableofcontents

\section{Introduction}

\subsection{Sums of Hermitian matrices, Horn probability measures, and (augmented) hives}

Throughout this paper, $n \geq 2$ is a fixed dimension (which one should think of as being large).

Let $\Spec$ denote the cone of all possible tuples $x$ in $\R^n$ such that 
$$x_1 \geq x_2 \geq \dots \geq x_n.$$
Here and in the sequel, whenever we use a symbol such as $x$ to denote a vector in $\R^n$, we use $x_1,\dots,x_n$ to denote its components, and similarly for other symbols such as $\lambda, \mu, \nu, \pi, \sigma, a$, etc..  We let $\Spec^\circ$ denote the interior of $\Spec$, that is to say the cone of tuples $x = (x_1, \dots, x_n)$ with
$$x_1 > x_2 > \dots > x_n.$$
Note that the eigenvalues of an $n \times n$ Hermitian matrix, ordered in non-increasing order, become an element of $\Spec$.  Let us define a relation
\begin{equation}\label{rel}
 \lambda \boxplus \mu \rel \nu
\end{equation}
if there exist Hermitian matrices $A,B$ with eigenvalues $\la, \mu$ respectively such that $A+B$ has eigenvalues $\nu$.  Thus for instance $(5,3) \boxplus (3,0) \rel (6,5)$ since $\diag(5,3) + \diag(0,3) = \diag(5,6)$, but also $(5,3) \boxplus (3,0) \rel (8,3)$ since $\diag(5,3) + \diag(3,0) = \diag(8,3)$.  Here one should view $\boxplus, \rel$ as formal symbols\footnote{One can also interpret $\boxplus$ as a hypergroup operation or as a special case of orbital convolution; see e.g., \cite{dooley}.} rather than binary operations or relations, though they are of course suggestively written to invoke analogies with the addition operation $+$ and equality relation $=$ respectively.

 In \cite{Weyl}, Weyl asked the question of determining necessary and sufficient conditions on $\la, \mu, \nu \in \Spec$ for the relation \eqref{rel} to hold.  As conjectured by Horn \cite{Horn} and proven in \cite{KT1}, the set 
$$ \HORN_{\la \boxplus \mu} \coloneqq \{ \nu \in \Spec: \lambda \boxplus \mu \rel \nu \}$$
of possible $\nu$ arising from a given choice of $\la, \mu$ forms a polytope (known as the \emph{Horn polytope}), given by the trace condition
\begin{equation}\label{trace}
 \sum \lambda + \sum \mu = \sum \nu
\end{equation}
(where we abbreviate $\sum \lambda \coloneqq \sum_{i=1}^n \lambda_i$)
together with a recursively defined set of linear inequalities known as the \emph{Horn inequalities}, which include for instance the Weyl inequalities
$$ \nu_{i+j-1} \leq \lambda_{i} + \mu_{j}$$
for $1 \leq i,j,i+j-1 \leq n$, as well as many others.  We refer the reader to \cite{KT2} for a survey of the history of this problem and its resolution.  For $\lambda,\mu \in \Spec^\circ$, the Horn polytope $\HORN_{\la \boxplus \mu}$ is $n-1$-dimensional.

One of the key tools used in the proof of the Horn conjecture in \cite{KT1} is that of a \emph{hive}\footnote{There is also the closely related notion of a \emph{honeycomb}, which we will not use in this paper.}, defined as follows.

\begin{definition}[Hives] A \emph{rhombus} is a quadruple $ABCD$ in the lattice\footnote{It often will be geometrically natural to depict $\Z^2$ as a triangular lattice rather than a rectangular one, though the precise choice of representation of $\Z^2$ as a triangular lattice may vary depending on context.}
 $\Z^2$ of one of the following three forms for some $i,j \in \Z^2$:
\begin{itemize}
\item[(i)] $(A,B,C,D) = ((i,j), (i+1,j), (i+2,j+1), (i+1,j+1))$
\item[(ii)] $(A,B,C,D) = ((i,j), (i+1,j+1), (i+1,j+2), (i,j+1))$
\item[(iii)] $(A,B,C,D) = ((i,j), (i,j-1), (i+1,j-1), (i+1,j))$.
\end{itemize}
We refer to $AC$ as the \emph{long diagonal} of the rhombus and $BD$ as the \emph{short diagonal}; see Figure \ref{fig:hive-schematic}.  A function $h \colon \Omega \to \R$ defined on some subset $\Omega$ of $\Z^2$ is said to be \emph{rhombus-concave} if one has
\begin{equation}\label{rhombus-concave}
 h(A) + h(C) \leq h(B) + h(D)
\end{equation}
for all rhombi $ABCD$ in $\Omega$.  A \emph{hive} is a rhombus concave function $h \colon T \to \R$ defined on the triangle
\begin{equation}\label{T-def}
 T \coloneqq \{ (i,j) \in \Z^2: 0 \leq i \leq j \leq n \}.
\end{equation}
If $\la, \mu, \nu \in \Spec$, we say that a hive $h$ has boundary condition $\lambda \boxplus \mu \rel \nu$ if one has
\begin{align}
 h(0,i) &= \sum_{j=1}^i \la_j \label{boundary-1}\\
 h(i,n) &= \sum \la + \sum_{j=1}^i \mu_j \label{boundary-2}\\
 h(i,i) &= \sum_{j=1}^i \nu_j \label{boundary-3}
\end{align}
for all $0 \leq i \leq n$, and write $\HIVE_{\la \boxplus \mu \rel \nu}$ for the set of all hives with boundary condition
$\lambda \boxplus \mu \rel \nu$ (thus for instance $\HIVE_{\la \boxplus \mu \rel \nu}$ will be empty if \eqref{trace} fails, as the boundary conditions \eqref{boundary-2}, \eqref{boundary-3} then become inconsistent at $i=n$).  These boundary conditions are depicted schematically in Figure \ref{fig:hive-schematic}.  We also adopt the ``wildcard convention'' that replacing a tuple such as $\la$, $\mu$, or $\nu$ with an asterisk $\ast$ denotes the operation of taking unions over all values of that tuple, thus for instance
\begin{align*}
\HIVE_{\la \boxplus \mu \rel \ast} &\coloneqq \bigcup_\nu \HIVE_{\la \boxplus \mu \rel \nu} \\
\HIVE_{\la \boxplus \ast \rel \nu} &\coloneqq \bigcup_\mu \HIVE_{\la \boxplus \mu \rel \nu} \\
\HIVE_{\ast \boxplus \ast \rel \ast} &\coloneqq \bigcup_{\lambda,\mu,\nu} \HIVE_{\la \boxplus \mu \rel \nu} 
\end{align*}
denote the hives with boundary conditions $\lambda \boxplus \mu \rel \ast$, $\lambda \boxplus \ast \rel \nu$,and $\ast \boxplus \ast \rel \ast$ respectively.  Note that all of these sets are convex.
\end{definition}

\begin{figure}
\begin{center}
\includegraphics[scale=0.40]{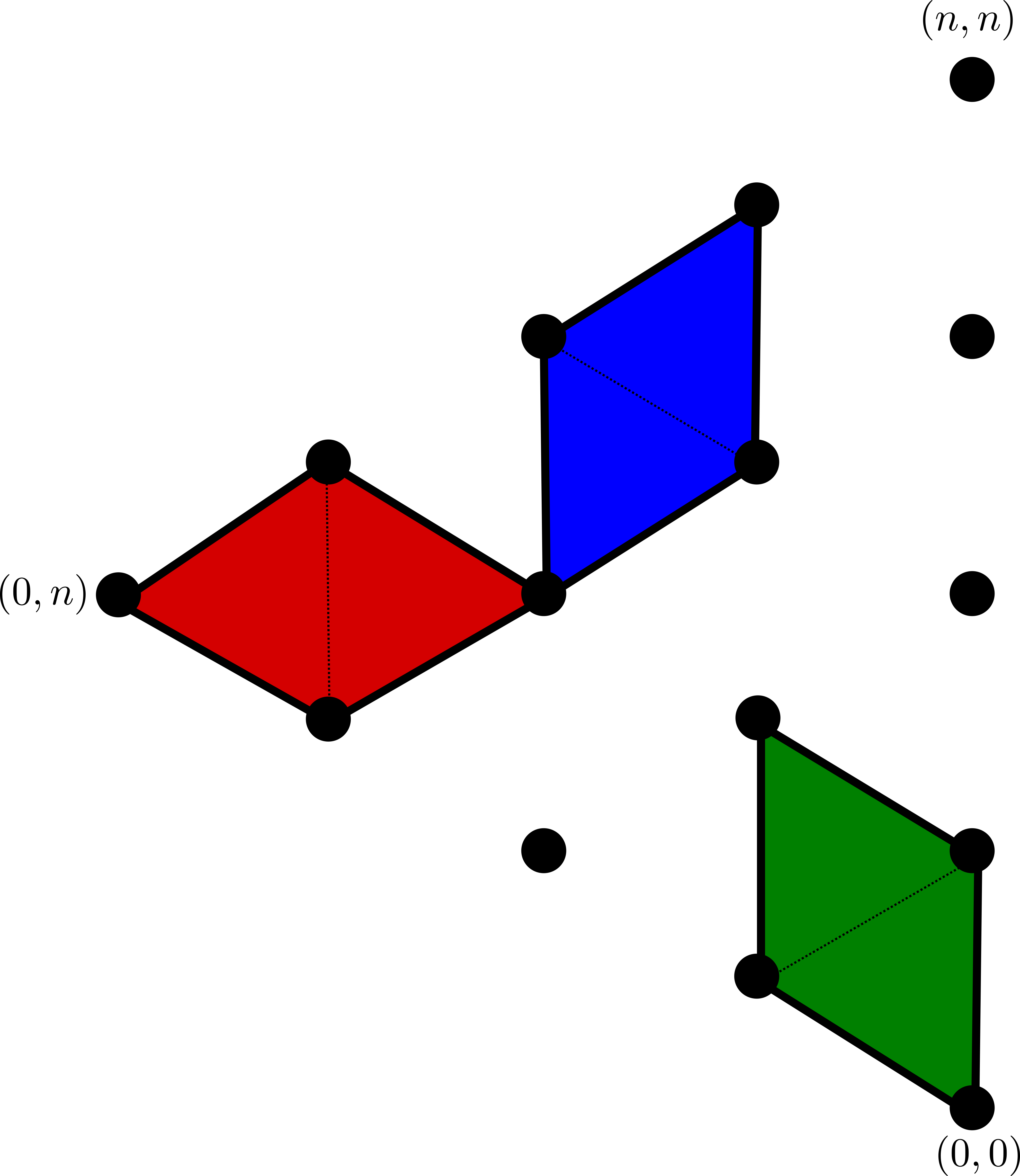}
\caption{The triangular region $T$ with $n=4$, tilted to lie on equilateral lattice (so that a rhombus is precisely the union of two adjacent unit equilateral triangles).  The blue, green, and red regions are rhombi of the form (i), (ii), (iii) respectively, with the dotted lines representing the short diagonals.}\label{fig:rhombus}
\end{center}
\end{figure}

\begin{figure}
\begin{center}
\includegraphics[scale=0.40]{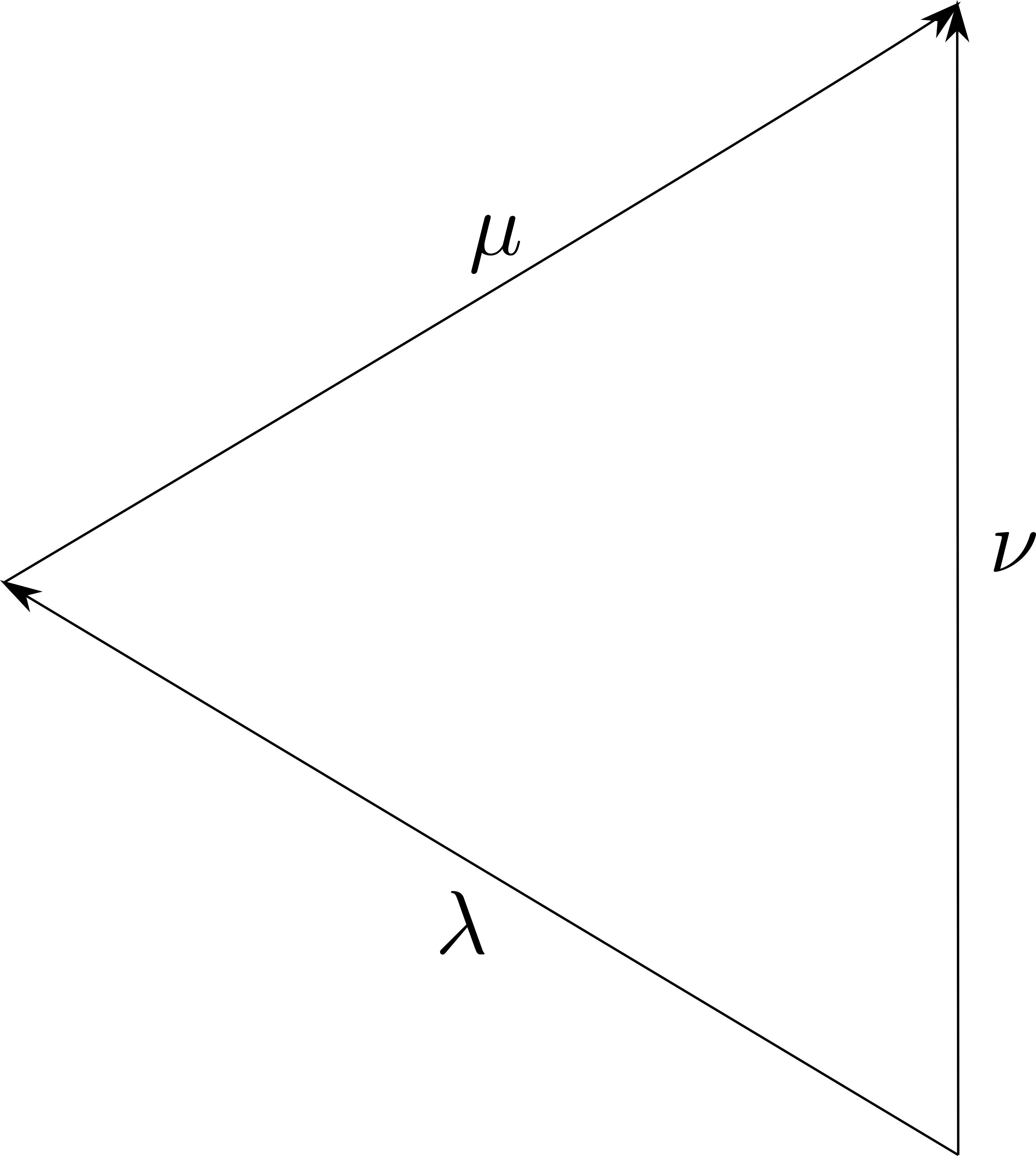}
\caption{A schematic depiction of the boundary condition $\lambda \boxplus \mu \rel \nu$, using the same orientation used in Figure \ref{fig:rhombus}.  Thus, the hive increases according to the tuple $\lambda$ as one moves from the southern vertex $(0,0)$ to the western one $(0,n)$, according to the tuple $\mu$ as one moves from the western vertex $(0,n)$ to the northern vertex $(n,n)$, and according to the tuple $\nu$ as  one moves from the southern vertex $(0,0)$ to the northern vertex $(n,n)$.}\label{fig:hive-schematic}
\end{center}
\end{figure}

\begin{figure}
\begin{center}
\includegraphics[scale=0.40]{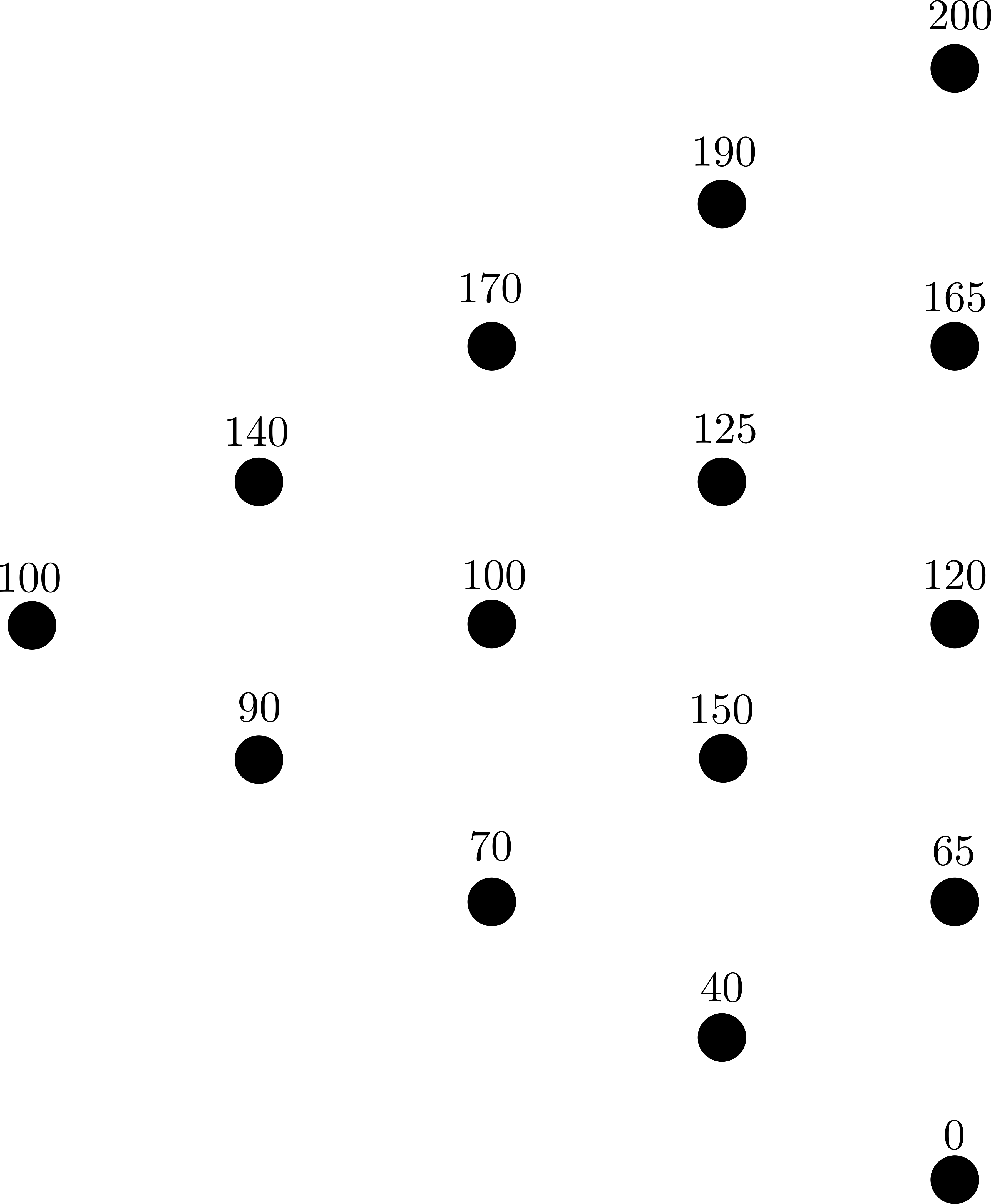}
\caption{A hive with boundary condition $(40, 30, 20, 10) \boxplus (40,30,20,10) \rel (65,55,45,35)$.}\label{fig:hive-example}
\end{center}
\end{figure}

The key relationship between hives and the Weyl problem is given by

\begin{proposition}\label{weyl-prop}  Let $n \geq 1$ and $\la, \mu, \nu \in \Spec$.  Then $\la \boxplus \mu \rel \nu$ if and only if $\HIVE_{\la \boxplus \mu \rel \nu}$ is non-empty.  Equivalently, $\HORN_{\la \boxplus \mu}$ is the projection of $\HIVE_{\la \boxplus \mu \rel \ast}$ under the linear map that sends $\HIVE_{\la \boxplus \mu \rel \nu}$ to $\nu$ for all $\nu$.
\end{proposition}

\begin{proof}  See \cite[Appendix 8]{KTW} and \cite[Appendix 2]{KT1}.  An alternate proof is provided in \cite{Speyer-Horn}.
\end{proof}

We remark that for $\la, \mu \in \Spec^\circ_n$, $\HIVE_{\lambda \boxplus \mu \rel \nu}$ is $\binom{n-1}{2}$-dimensional for $\nu$ in the interior of $\HORN_{\lambda \boxplus \mu}$, and $\HIVE_{\la \boxplus \mu \rel \ast}$ is $\binom{n}{2}$-dimensional.  Intuitively, the hive polytope $\HIVE_{\la \boxplus \mu \rel \nu}$ represents the possible ways in which the relation $\la \boxplus \mu \rel \nu$ can hold.

\begin{remark} There is a discrete analogue of the above theory (in the spirit of the Kirillov orbit method): when $\lambda,\mu,\nu$ take values in the non-negative integers, then the \emph{Littlewood--Richardson coefficient} $c^\nu_{\lambda \mu}$ is precisely the number of lattice points in $\HIVE^\nu_{\lambda \mu}$, and the \emph{saturation conjecture} established in \cite{KT1} asserts that this number of lattice points is non-zero if and only if $\HIVE^\nu_{\lambda \mu}$ is non-empty.  These coefficients arise in many contexts, including the tensor product multiplicities of $\GL_n$ representations and the intersection numbers for Schubert calculus. We refer to \cite{KT2} for further discussion, and \cite{Nar, Okounkov, Greta} for some exploration of the relation between the volume of the hive polytope and the number of its lattice points.  However, we will not study these coefficients further here.
\end{remark}

There is a natural probability measure on the Horn polytope $\HORN_{\la \boxplus \mu}$, referred to as the \emph{Horn probability measure} in \cite{Zuberhorn}, defined as the eigenvalues of $A+B$ when $A, B$ are chosen independently and uniformly (i.e., with respect to $U(n)$-invariant Haar measure) the space (essentially a coadjoint orbit) of all Hermitian matrices with eigenvalues $\la, \mu$ respectively.  This Horn measure turns out to be piecewise polynomial and was computed explicitly in \cite[(8), Proposition 4]{Zuber} (see also \cite{Zuberhorn}) to be given by the formula
\begin{equation}\label{horn-prob}
\frac{V(\nu) V(\tau)}{V(\la) V(\mu)} |\HIVE_{\la \boxplus \mu \rel \nu}| \ d\nu
\end{equation}
for $\la, \mu \in \Spec^\circ$, where $|\HIVE_{\lambda \boxplus \mu \rel \nu}|$ denotes the $\binom{n-1}{2}$-dimensional Lebesgue measure of the hive polytope $\HIVE_{\lambda \boxplus \mu \rel \nu}$, $d\nu = d\nu_1 \dots d\nu_{n-1}$ is $n-1$-dimensional Lebesgue measure on the hyperplane given by \eqref{trace}, 
$$ V(\la) = V_n(\la) \coloneqq \prod_{1 \leq i < j \leq n} (\la_i - \la_j)$$
is the Vandermonde determinant, and $\tau$ is the tuple
$$ \tau \coloneqq ( n, n-1, \dots, 1 ).$$

The factor of $V(\nu)$ in \eqref{horn-prob} is inconvenient, but can be removed through the device of \emph{Gelfand--Tsetlin patterns}, which can be viewed as a limiting (and much better understood) case of a hive.  Analogously to \eqref{rel}, we introduce the relation
$$ \diag(\la) \rel a$$
for $\la \in \Spec$ and $a \in \R^n$ to denote the claim that there exists a Hermitian matrix $A$ with eigenvalues $\la$ and diagonal entries $a_1,\dots,a_n$.  The classical \emph{Schur--Horn theorem} \cite{schur}, \cite{horn} asserts that the relation $\diag(\la) \rel a$ holds if and only if \emph{majorized} by $\lambda$ in the sense that one has the trace condition
\begin{equation}\label{al-trace}
\sum a = \sum \lambda
\end{equation}
and the majorizing inequalities
\begin{equation}\label{major}
 a_{i_1} + \dots + a_{i_k} \leq \lambda_1 + \dots + \lambda_k
\end{equation}
for all $1 \leq i_1 < \dots < i_k \leq n$; equivalently, $a$ lies in the \emph{permutahedron} formed by the convex hull of the image of $\lambda$ under the permutation group $S_n$.

Now define a Gelfand--Tsetlin pattern to be a pattern $\gamma = (\lambda_{j,k})_{1 \leq j \leq k \leq n}$ of real numbers obeying the interlacing conditions
\begin{equation}\label{interlacing}
 \lambda_{j,k+1} \geq \lambda_{j,k} \geq \lambda_{j+1,k+1}
\end{equation}
for $1 \leq j \leq k \leq n-1$; see Figure \ref{fig:gt}.  We say that this pattern has boundary condition $\diag(\lambda) \rel a$ for some $\lambda \in \Spec_n$ and $a \in \R^n$ if one has
$$ \lambda_{j,n} = \lambda_j$$
for $1 \leq j \leq n$ and
$$ \sum_{j=1}^k \lambda_{j,k} = \sum_{j=1}^k a_j$$
for $1 \leq k \leq n$; see Figure \ref{fig:gt-example} for an example.  The polytope of all Gelfand--Tsetlin patterns with boundary condition $\diag(\lambda) \rel a$ will be denoted $\GT_{\diag(\lambda) \rel a}$, and we adopt the same wildcard convention as before, thus for instance
$$\GT_{\diag(\lambda) \rel \ast} \coloneqq \bigcup_a \GT_{\diag(\lambda) \rel a}$$
and
$$\GT_{\diag(\ast) \rel \ast} \coloneqq \bigcup_\lambda \GT_{\diag(\lambda) \rel \ast}$$
denote the sets of Gelfand--Tsetlin patterns with boundary conditions $\diag(\lambda) \rel \ast$ and $\diag(\ast) \rel \ast$ respectively.  We remark that for $\lambda \in \Spec^\circ_n$, $\GT_{\diag(\lambda) \rel \ast}$ is a $\binom{n}{2}$-dimensional polytope, which we call a \emph{Gelfand--Tsetlin polytope}, while $\GT_{\diag(\ast) \rel \ast}$ is a $\binom{n+1}{2}$-dimensional convex cone, which we call the \emph{Gelfand--Tsetlin cone}.

\begin{figure}
\begin{center}
\includegraphics[scale=0.50]{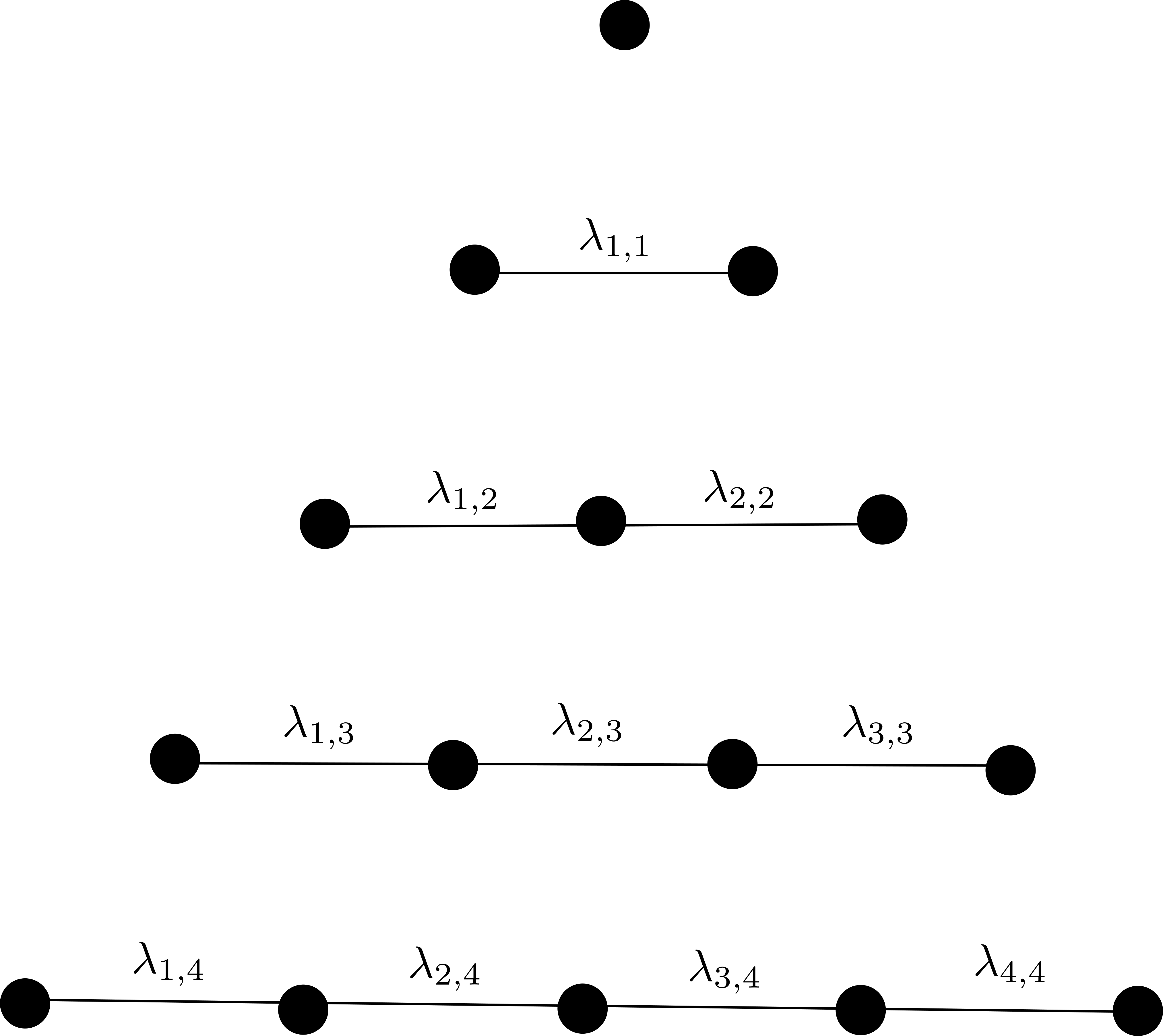}
\caption{An $n=4$ Gelfand--Tsetlin pattern.  Each number $\lambda_{i,j}$ in the pattern is greater than or equal to numbers immediately to the northeast or southeast of the pattern; in particular, every row of the pattern is decreasing. Note that such patterns are sometimes depicted as inverted pyramids instead of pyramids in the literature.}\label{fig:gt}
\end{center}
\end{figure}

\begin{figure}
\begin{center}
\includegraphics[scale=0.50]{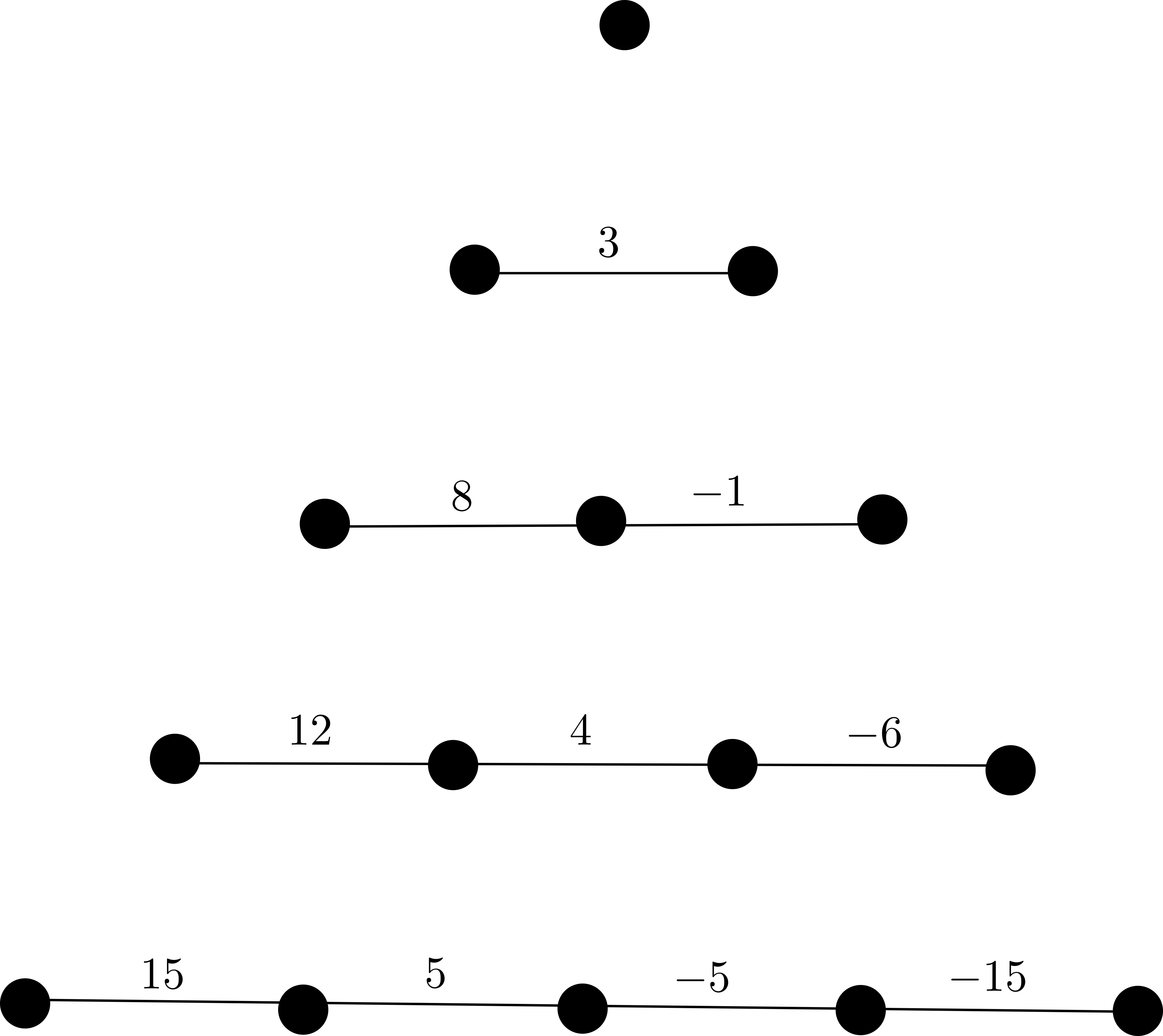}
\caption{A Gelfand--Tsetlin pattern with boundary $\diag(15,5,-5,-15) \rel (3,4,3,-10)$.}\label{fig:gt-example}
\end{center}
\end{figure}

We recall the following standard facts about Gelfand--Tsetlin polytopes:

\begin{proposition}\label{gt-rem}  Let $\lambda \in \Spec^\circ$.
\begin{itemize}
\item[(i)] If $a \in \R^n$, then $\diag(\lambda) \rel a$ holds if and only if $\GT_{\diag(\lambda) \rel a}$ is non-empty.
\item[(ii)]  The $\binom{n}{2}$-dimensional volume of $\GT_{\diag(\lambda) \rel \ast}$ is $V(\lambda) / V(\tau)$.
\item[(iii)]  Let $A$ be a random Hermitian matrix with eigenvalues $\lambda$, drawn using the $U(n)$-invariant Haar probability measure.  For $1 \leq k \leq n$, let $\lambda_{1,k} \geq \dots \geq \lambda_{k,k}$ be eigenvalues of the top left $k \times k$ minor of $A$.  Then $(\lambda_{j,k})_{1 \leq j \leq k \leq n}$ lies in the polytope $\GT_{\diag(\lambda) \rel \ast}$ with the uniform probability distribution; it has boundary data $\diag(\lambda) \rel a$ where $a = (a_{11},\dots,a_{kk})$ are the diagonal entries of $A$.
\item[(iv)]  If $\Lambda \in \Spec$ has \emph{large gaps} in the sense that 
\begin{equation}\label{large-gaps}
\min_{1 \leq i < n} \Lambda_i - \Lambda_{i+1} > \lambda_1 - \lambda_n,
\end{equation}
then there is a volume-preserving linear bijection between $\GT_{\diag(\lambda) \rel a}$ and $\HIVE_{\Lambda \boxplus \lambda \rel \Lambda+a}$ for any $a \in \R^n$, with a Gelfand--Tsetlin pattern $(\lambda_{j,k})_{1 \leq j \leq k \leq n}$ being mapped to the hive $h \colon T \to \R$ defined by the formula
\begin{equation}\label{hij}
 h(i,j) = \Lambda_1 + \dots + \Lambda_j + \lambda_{1,j} + \dots + \lambda_{i,j};
\end{equation}
see Figure \ref{fig:gt-hive}.
\end{itemize}
\end{proposition}

\begin{proof} For (i), observe on the one hand that if a Gelfand--Tsetlin pattern has boundary $\diag(\lambda) \rel a$, then the majorization conditions \eqref{al-trace}, \eqref{major} hold, and conversely if $a$ is a permutation of $\lambda$ then one can form a Gelfand--Tsetlin pattern with boundary $\diag(\lambda) \rel a$ by using the permutation defining $a$ to successively delete elements of $\lambda_1,\dots,\lambda_n$.  The claim now follows from the Schur--Horn theorem and convexity.

For (ii), see for example, \cite[Theorem 15.1]{Postnikov}, \cite[Theorem 1.1]{Karola}, or \cite[Lemma 1.12]{GUEqueue}.  For (iii), see \cite[\S 3.9]{GUEqueue}.  For (iv), direct calculation shows that any Gelfand--Tsetlin pattern $(\lambda_{j,k})_{1 \leq j \leq k \leq n}$ with boundary condition $\diag(\lambda) \rel a$ generates a hive with boundary condition $\Sigma \boxplus \lambda \rel \Sigma+a$ by the formula \eqref{hij}, and conversely that every hive with boundary condition $\Sigma \boxplus \lambda \rel \Sigma+a$ arises from a unique Gelfand--Tsetlin pattern in this manner.  The linear transformation \eqref{hij} is unipotent in a suitable basis, hence volume-preserving. (See also \cite[\S 4.3]{Zuber-2} for a closely related calculation.)
\end{proof}

\begin{remark}  Proposition \ref{gt-rem}(iv) shows that the relation $\diag(\lambda) \rel a$ is a limiting case of the relation
$\Sigma \boxplus \lambda \rel \Sigma+a$ when $\Sigma$ has sufficiently large gaps.  This is related to the observation that if $A$ is a Hermitian matrix and $S$ is a diagonal matrix whose diagonal entries $\Sigma$ have large gaps, then the eigenvalues of $S+A$ are close to $\Sigma + \diag(A)$.
\end{remark}

\begin{figure}
\begin{center}
\includegraphics[scale=0.50]{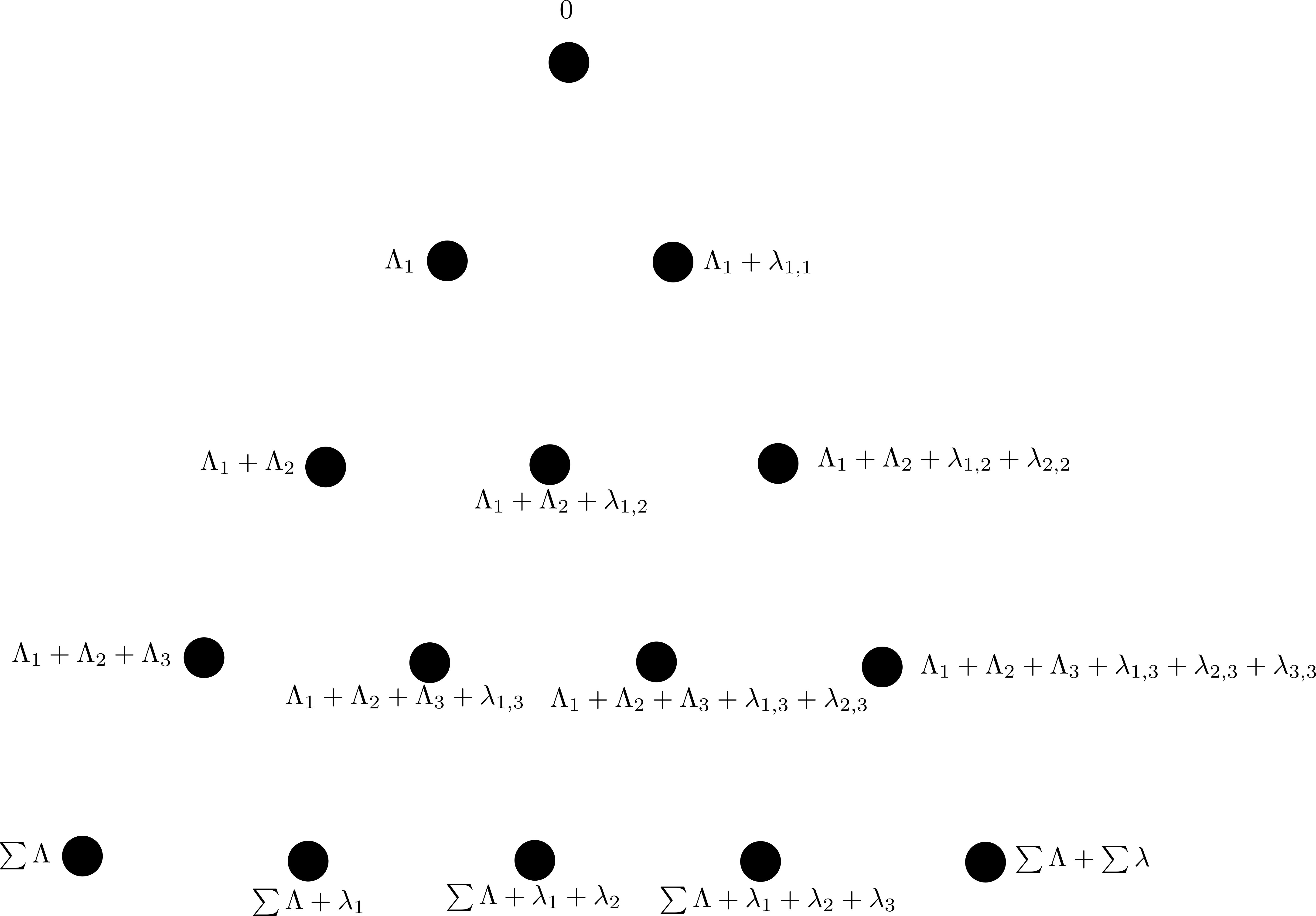}
\caption{The hive associated with the Gelfand--Tsetlin pattern in Figure \ref{fig:gt-example} and some large gap tuple $\Lambda$.  Note that the orientation of this hive is flipped and rotated from that in Figure \ref{fig:hive-example}.}\label{fig:gt-hive}
\end{center}
\end{figure}

\begin{figure}
\begin{center}
\includegraphics[scale=0.50]{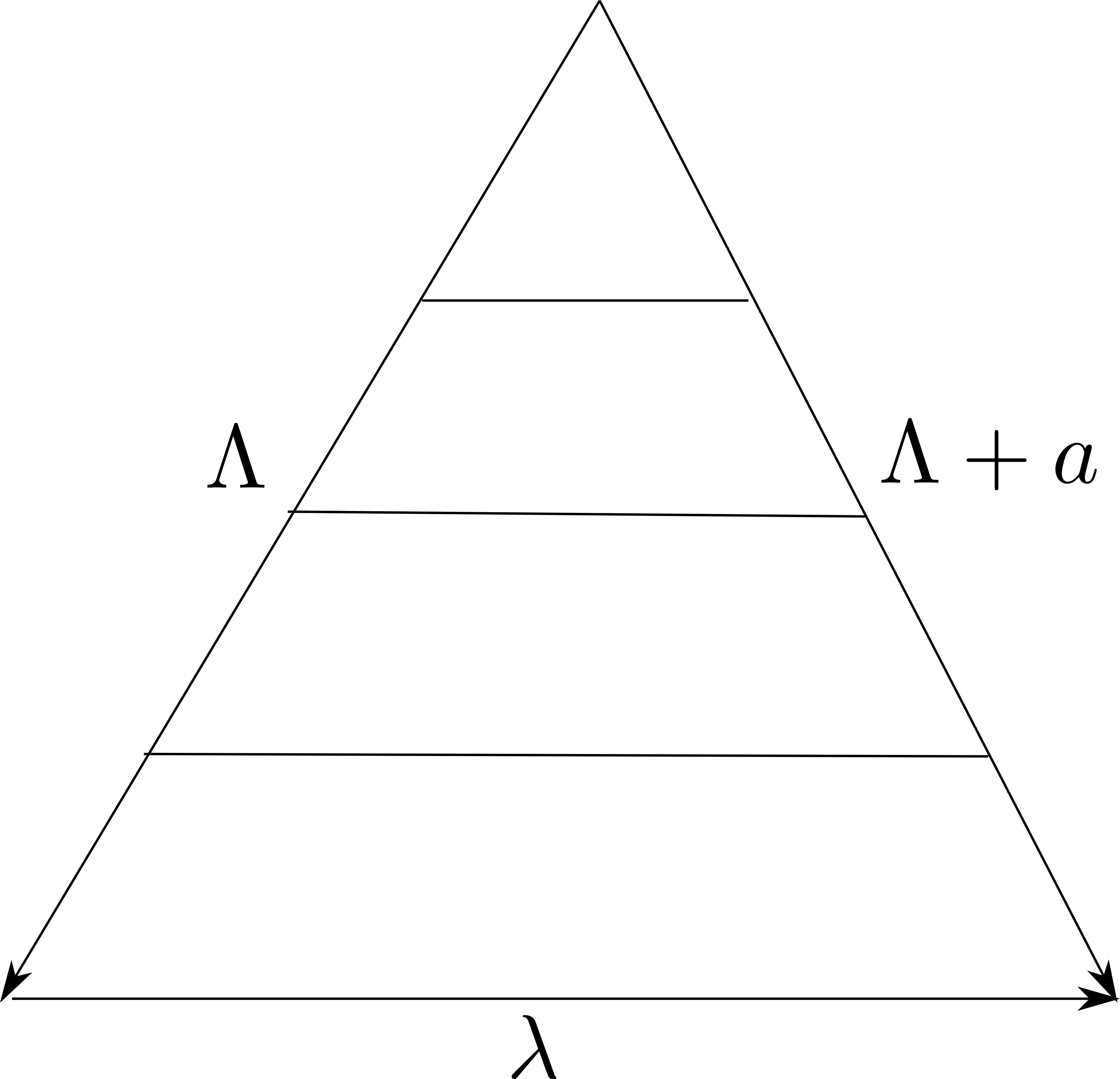}
\caption{A schematic depiction of the boundary conditions of the hive in Figure \ref{fig:gt-hive}.  The horizontal ``creases'' inside the triangle indicate that the rhombus concavity condition \eqref{rhombus-concave} is essentially an automatic consequence of the large gaps hypothesis for rhombi that cross these creases.}\label{fig:gt-hive-schematic}
\end{center}
\end{figure}

From this proposition, \eqref{horn-prob}, and the Fubini--Tonelli theorem, one can view the Horn probability measure associated to two tuples $\lambda,\mu \in \Spec^\circ$ to be the pushforward of $\frac{V(\tau)^2}{V(\la) V(\mu)}$ times Lebesgue measure on the ($(n-1)^2$-dimensional) \emph{augmented hive polytope}
$$ \AHIVE_{\diag(\lambda \boxplus \mu \rel \ast) \rel \ast} \coloneqq \bigcup_{\nu,a} \AHIVE_{\diag(\lambda \boxplus \mu \rel \nu) \rel a} \coloneqq \bigcup_{\nu,a} (\HIVE_{\lambda \boxplus \mu \rel \nu} \times \GT_{\diag(\nu) \rel a})$$
under the linear map that sends $\AHIVE_{\diag(\lambda\boxplus \mu \rel\nu) \rel a}$ to $\nu$ for each $\nu,a$.  We refer to the elements $(h, \gamma)$ of the augmented hive polytope $\AHIVE_{\diag(\lambda \boxplus \mu \rel \ast) \rel \ast}$, as \emph{augmented hives}. 

Using Proposition \ref{gt-rem}(iv), one can view an augmented hive as two hives glued together along a common edge, where two of the boundaries have large gaps: see Figure \ref{fig:augment}.

\begin{figure}
\begin{center}
\includegraphics[scale=0.40]{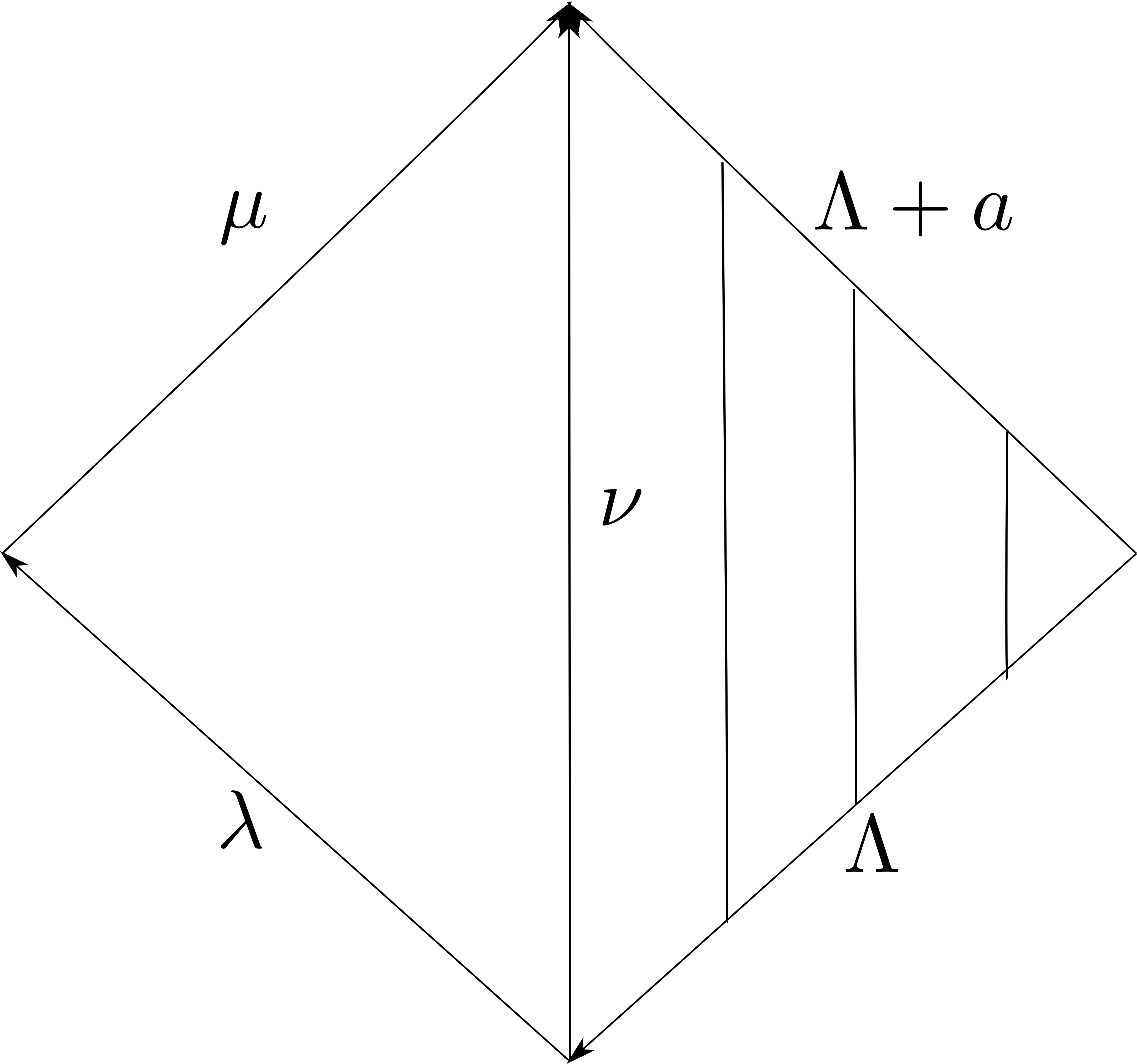}
\caption{A schematic depiction of an augmented hive in $\AHIVE_{\diag(\lambda \boxplus \mu \rel \nu) \rel a}$, where we artificially shift by a tuple $\Lambda$ with large gaps in order to create two hives, instead of a hive and a Gelfand--Tsetlin pattern.}\label{fig:augment}
\end{center}
\end{figure}

As just one illustration of the power of this characterization, we observe as an immediate corollary using Pr\'ekopa's theorem \cite{prekopa} (or the Pr\'ekopa--Leindler inequality, or the Brunn--Minkowski inequality) that the Horn probability measure is log-concave for any $\lambda,\mu \in \Spec^\circ$.

To summarize the discussion so far: the eigenvalues of the sum of two independent, uniformly distributed Hermitian matrices with eigenvalues $\lambda,\mu \in \Spec^\circ$ respectively, has the distribution of a linear projection of the uniform probability measure on the augmented hive polytope $\AHIVE_{\diag(\lambda \boxplus \mu \rel \ast) \rel \ast}$, which can be computed to be an $n(n-1)$-dimensional polytope.

\subsection{The large $n$ limit and GUE}

In principle, the behavior of the Horn probability measures in the limit $n \to \infty$ is governed by the theory of free probability: if the empirical distributions of $\la = \lambda^{(n)}, \mu = \mu^{(n)}$ converge (in an appropriate sense) to probability measures $\sigma, \sigma'$, then the empirical distribution of $\nu$ (drawn from the Horn probability measure) should similarly converge to the free convolution $\sigma  \boxplus \sigma'$; see for instance the seminal paper \cite{Voi} for some rigorous results in this direction.  Relating to this, results have emerged in recent years establishing large deviation inequalities for the Horn probability measure under suitable hypotheses on $\lambda, \mu$: see \cite{bgh}, \cite{NarSheff}.

The question of understanding the spectrum of $X_n + Y_n,$ in the setting of large deviations was studied  in \cite{bgh}, where upper and lower large deviation bounds were given which agreed for measures of a certain class that correspond to ``free products with  amalgamation" (see Theorem 1.3, \cite{bgh}). 

Suppose $\lambda^{cont}:[0,1]\rightarrow \R$ and $\mu^{cont}:[0,1]\rightarrow \R$ are $C^1$, strongly decreasing functions. Let $\lambda^{(n)}$ and $\mu^{(n)}$ be obtained from it by taking the slopes of the respective piecewise linear approximations to $\lambda^{cont}$ and $\mu^{cont}$, where the number of pieces is $n$.   In \cite[Theorem 8]{NarSheff}, a large deviation principle was obtained for the probability measure (as $n \ra \infty$) of the piecewise linear extension of $\frac{h_n}{n^2}$ to $T$, 
where $(h_n, \gamma_n)$ is an augmented hive sampled uniformly at random from $\AHIVE_{\diag(\lambda^{(n)} \boxplus \mu^{(n)} \rel \ast) \rel \ast}$.

Comparing these results with the previous discussion, it is natural to ask if Lebesgue measure on the polytope $\AHIVE_{\diag(\lambda \boxplus \mu \rel \ast) \rel \ast}$ also exhibits concentration.  As a first step towards this goal, we are able to establish this for spectra $\lambda,\mu$ that are not deterministic, but are instead drawn from (scalar multiples of) the \emph{GUE ensemble}.  To establish normalization conventions, we define a GUE random matrix to be a random Hermitian matrix $M = (\xi_{ij})_{1 \leq i,j \leq n}$ where $\xi_{ij} = \overline{\xi_{ji}}$ for $i<j$ are independent complex gaussians of mean zero and variance $1$, $\xi_{ii}$ are independent real gaussians of mean zero and variance $1$, independent of the $\xi_{ij}$ for $i<j$.  As is well known (see e.g., \cite{mehta}), if $\sigma>0$ and $A$ is a random matrix with $\frac{A}{\sqrt{\sigma^2 n}}$ drawn from the GUE ensemble, then the eigenvalues $\lambda \in \Spec$ of $A$ are distributed with probability density function
\begin{equation}\label{density}
 C_n \sigma^{-\frac{n(n+1)}{2}} \exp\left( - \frac{|\lambda|^2}{2\sigma^2 n} \right) V(\lambda)^2
\end{equation}
for some constant $C_n>0$ depending only on $n$.  In particular, $\lambda$ will lie in $\Spec^\circ$ almost surely. From this the previous discussion, we see that if $\sigma_\lambda, \sigma_\mu > 0$ are fixed\footnote{In particular, we allow implied constants in the $O()$ notation to depend on these quantities.} and $A, B$ are independent random matrices with\footnote{This normalization is chosen so that the mean eigenvalue gaps of $A,B$ comparable to $1$ in the bulk of the spectrum.} $\frac{A}{\sqrt{\sigma_\lambda^2 n}}, \frac{B}{\sqrt{\sigma_\mu^2 n}}$ drawn from the GUE ensemble, then the the distribution of the eigenvalues of $A+B$ are the pushforward of the measure 
on the $n(n+1)$-dimensional \emph{augmented hive cone} 
$$\AHIVE_{\diag(\ast \boxplus \ast \rel \ast) \rel \ast} \coloneqq \bigcup_{\lambda,\mu,\nu,\pi} (\HIVE_{\lambda \boxplus \mu \rel \nu} \times \GT_{\diag(\nu) \rel \pi}),$$ 
where the probability density function of this measure is given by
\begin{equation}\label{css}
 C_{n,\sigma_\lambda,\sigma_\mu} \exp\left( - \frac{|\lambda|^2}{2\sigma_\lambda^2 n} - \frac{|\mu|^2}{2\sigma_\mu^2 n} \right) V(\lambda) V(\mu)
\end{equation}
on the slices
\begin{equation}\label{ahive-slice}
\AHIVE_{\diag(\lambda \boxplus \mu \rel \ast) \rel \ast} \coloneqq \bigcup_{\nu,\pi} (\HIVE_{\lambda \boxplus \mu \rel \nu} \times \GT_{\diag(\nu) \rel \pi})
\end{equation}
and where $C_{n,\sigma_\lambda,\sigma_\mu} > 0$ is chosen to make this measure a probability measure.

Since GUE matrices have an operator norm of $O(\sqrt{n})$ with overwhelming probability (by which we mean with probability $1-O(n^{-C})$ for any fixed $C>0$), the boundary differences $\lambda, \mu, \nu$ of an augmented hive $(h,\gamma)$ drawn from the above measure will be of size $O(n)$ with overwhelming probability, and hence the entries $h(v), v \in T$ of the hive will be of size $O(n^2)$ with overwhelming probability.  From this fact (and some crude moment estimates to treat the contribution of the exceptional event), it is not difficult to show the ``trivial bound'' that the variance $\Var h(v)$ of any individual entry $h(v)$ of the hive is bounded by $O(n^4)$.  Of course, here and in the sequel we use the asymptotic notation $X \ll Y$, $Y \gg X$, or $X = O(Y)$ to denote the estimate $|X| \leq CY$ for some constant $C$ (depending only on $\sigma_\lambda, \sigma_\mu$); we also use $X \asymp Y$ for $X \ll Y \ll X$, and $X = o(Y)$ to denote the estimate $|X| \leq c(n) Y$ for some $c(n)$ that goes to zero as $n$ goes to infinity (keeping $\sigma_\lambda, \sigma_\mu$ fixed).

The main result of this paper is a gain over this trivial bound:

\begin{theorem}[Concentration of augmented hives]\label{mainthm}  Let $\sigma_\lambda,\sigma_\mu > 0$ be fixed, and let $(h,\gamma) \in \AHIVE_{\diag(\ast \boxplus \ast \rel \ast) \rel \ast}$ be a random augmented hive drawn using the probability measure \eqref{css}.  Then for all $v \in T$, we have the variance bound
$$ \Var h(v) = o( n^4 )$$
as $n \to \infty$, uniformly in $v$.
\end{theorem}

Informally, this theorem asserts that randomly selected hives (with GUE boundary data) have an asymptotic limiting profile, at least in a subsequential sense.  It would be of interest to determine the uniqueness of this limiting profile and what this limiting profile is; most likely it will be related to the surface tension function introduced in \cite{random_concave} (see also \cite{NarSheff}).

\begin{remark}  Since the probability density function \eqref{css} of $v$ is log-concave, the scalar random variable $h(v)$ also has a log-concave density thanks to Pr\'ekopa's theorem \cite{prekopa}.  As is well known (see e.g., \cite{ALLOP}), control on the variance of a log-concave function implies subexponential tail bounds.  In particular, from Theorem \ref{mainthm} we can conclude that for every $\eps>0$ one has
$$ \E \exp\left( \frac{|h(v) - \E h(v)|}{\eps n^2} \right) \leq 2$$
(say) for $n$ sufficiently large depending on $\eps$.
\end{remark}

\subsection{Methods of proof}

We now discuss the main steps in the proof of Theorem \ref{mainthm}.  The first step is to exploit the \emph{octahedron recurrence}, which has appeared in the enumerative combinatorics literature several times (dating back at least to \cite{rr}), and which was observed in \cite{KTW} to witness an ``associativity'' property 
\begin{equation}\label{lamu}
\bigcup_\nu \HIVE_{\la \boxplus \mu \rel \nu} \times \HIVE_{\gamma \boxplus \nu \rel \pi}
\equiv \bigcup_\sigma \HIVE_{\gamma \boxplus \la \rel \sigma} \times \HIVE_{\sigma \boxplus \mu \rel \pi}
\end{equation}
on hives related to the trivial associativity 
\begin{equation}\label{abc}
(A+B)+C=A+(B+C)
\end{equation}
of the addition operation on Hermitian matrices $A,B,C$.  Indeed, from Proposition \ref{weyl-prop} and \eqref{abc} we see that the left-hand side of \eqref{lamu} is non-empty if and only if the right-hand side is, and in fact both polytopes have the same volume.  Informally, the relation \eqref{lamu} can be depicted schematically as $\gamma \boxplus (\la \boxplus \mu) \equiv (\gamma \boxplus \la) \boxplus \mu$.

In our context (viewing Gelfand--Tsetlin patterns as degenerations of hives), the octahedron recurrence is a piecewise-linear volume-preserving bijection
$$ \oct \colon \GT_{\diag(\ast) \rel \ast} \times \GT_{\diag(\ast) \rel \ast} \to \AHIVE_{\diag(\ast \boxplus \ast \rel \ast) \rel \ast}$$
between the two $n(n-1)$-dimensional convex cones $\GT_{\diag(\ast) \rel \ast} \times \GT_{\diag(\ast) \rel \ast}$, $\AHIVE_{\diag(\ast \boxplus \ast \rel \ast) \rel \ast}$.  In fact $\oct$ is a piecewise-linear volume-preserving bijection
\begin{equation}\label{oct-map}
 \oct \colon \bigcup_b \GT_{\diag(\la) \rel b} \times \GT_{\diag(\mu) \rel a-b} \to \bigcup_\nu \AHIVE_{\diag(\la \boxplus \mu \rel \nu) \rel a}
\end{equation}
for any $\la, \mu \in \Spec^\circ$ and $a \in \R^d$.  This map is related to the trivial linearity relation
\begin{equation}\label{diag-ident}
 \diag(A+B) = \diag(A) + \diag(B)
\end{equation}
of the operation $\diag$ of extracting the diagonal elements $\diag(A)$ of a Hermitian matrix $A$; this linearity is in some sense a degeneration of the associativity \eqref{abc} in the limit where $C$ is diagonal and has eigenvalue gaps going to infinity.  Indeed, we see from \eqref{diag-ident} and Proposition \ref{gt-rem}(i) that the domain in \eqref{oct-map} is non-empty if and only if the range is. Informally, the relation \eqref{oct-map} can be depicted schematically as $\diag(\la) + \diag(\mu) \equiv \diag(\la \boxplus \mu)$.

\begin{figure}
\begin{center}
\includegraphics[scale=0.40]{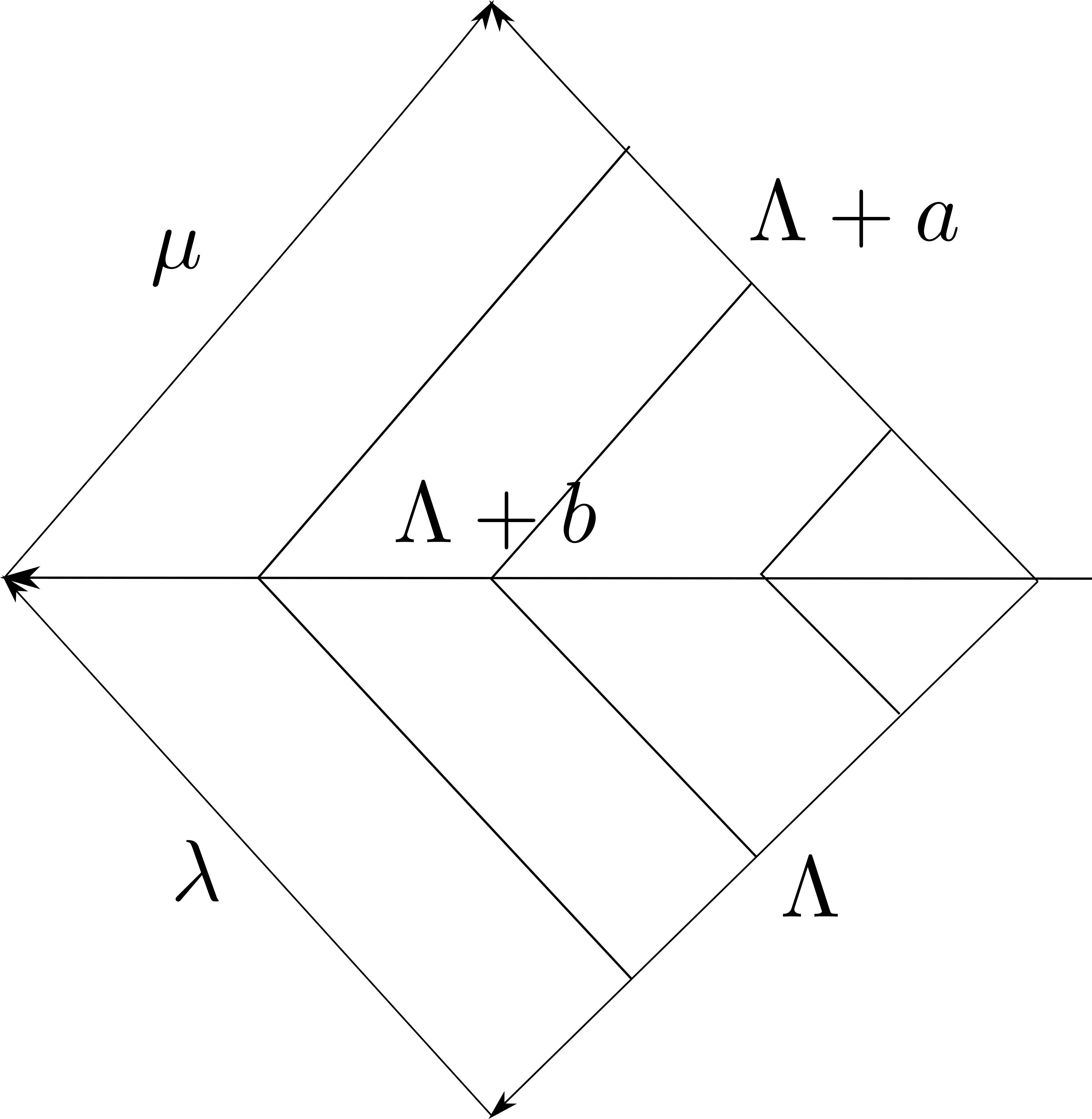}
\caption{A schematic depiction of a pair in $\GT_{\diag(\la) \rel b} \times \GT_{\diag(\mu) \rel a-b}$, where an artificial shift by a tuple $\Lambda$ with large gaps is used to re-interpret this pair as a pair of hives with a common edge.  The octahedron recurrence transforms such a pair of hives to another pair of hives of the form indicated in Figure \ref{fig:augment}.}\label{fig:excavated}
\end{center}
\end{figure}

In the work of Speyer \cite{Speyer}, a useful explicit form of the octahedron recurrence $\oct$ is obtained.  We record one particular consequence of this form here, proven in Section \ref{recurrence-sec}:

\begin{theorem}[Excavation form of octahedron recurrence]\label{excav}  Let $v$ be an element of the triangle $T$.  Then there is an explicit finite family ${\mathcal W}_{v} \colon \GT_{\diag(\ast) \rel \ast} \times \GT_{\diag(\ast) \rel \ast} \to \R$ (defined in Section \ref{recurrence-sec}) of linear functionals on $\GT_{\diag(\ast) \rel \ast} \times \GT_{\diag(\ast) \rel \ast}$, such that whenever $(h,g) = \oct(g_1,g_2)$ is the image of the octahedron recurrence for some $g_1, g_2 \in \GT_{\diag(\ast) \rel \ast}$, then
\begin{equation}\label{hv}
 h(v) = \max_{w \in {\mathcal W}} w( g_1, g_2 ).
\end{equation}
\end{theorem}

The linear functionals are given in terms of lozenge tilings of a certain hexagon $\hexagon_v$ associated to $v$; this version of the Speyer formula has not explicitly been written previously in the literature, and may be of independent interest.

From the volume-preserving nature of the octahedron recurrence and Fubini's theorem, we see that the probability distribution on $\AHIVE_{\diag(\ast \boxplus \ast \rel \ast) \rel \ast}$ with density function \eqref{css} is the pushforward under $\oct$ of the probability distribution on the product cone $\GT_{\diag(\ast) \rel \ast} \times \GT_{\diag(\ast) \rel \ast}$, with the density function given by the same formula \eqref{css} as before on each slice $\GT_{\diag(\lambda) \rel \ast} \times \GT_{\diag(\mu) \rel \ast}$.  By Proposition \ref{gt-rem}, this latter distribution can also be viewed as the distribution of pairs 
\begin{equation}\label{minor}
(\gamma_1,\gamma_2) = \left( (\lambda_{j,k})_{1 \leq j \leq k \leq n}, (\mu_{j,k})_{1 \leq j \leq k \leq n}\right)
\end{equation}
where $\lambda_{j,k}$ (resp. $\mu_{j,k}$) is the $j^{\mathrm{th}}$ eigenvalue of the top left $k \times k$ minor of $A$ (resp. $B$), and $A,B$ are independent Hermitian matrices with $\frac{A}{\sqrt{\sigma_\lambda^2 n}}, \frac{B}{\sqrt{\sigma_\mu^2 n}}$ drawn from the GUE ensemble.

Direct calculation reveals that the density function \eqref{css} is log-concave.  The supremum in Theorem \ref{excav} can then be handled by the following tool, which may be of independent interest:

\begin{lemma}\lab{lem:17}
Let $\eta$ be an log-concave probability measure in $\R^d$ with finite second moments, and let ${\mathcal W}$ be a family of affine functions $w \colon \R^d \to \R$.  Then
$$\Var_\eta \left(\sup_{w \in {\mathcal W}_v} w\right) \ll \sup_{w \in {\mathcal W}} (\Var_\eta w) \log(2+d).$$
Here of course we use the probabilistic notation $\E_\eta w \coloneqq \int_{\R^d} w\ d\eta$ and $\Var_\eta w \coloneqq \E_\eta |w|^2 - |\E_\eta w|^2$.
\end{lemma}

This lemma is a consequence of Cheeger's inequality and recent work of Klartag \cite{Klartag-KLS} on the KLS conjecture \cite{KLS}; we give the details in Section \ref{log-concave-sec}.  In view of this lemma, it would now suffice to establish the variance bound
$$ \Var w(\gamma_1, \gamma_2) = O( n^{4-c} )$$
for all $v \in T$ and $w \in {\mathcal W}_v$ and some constant $c>0$, where $(\gamma_1,\gamma_2)$ was the random variable \eqref{minor}; the additional factor of $n^{-c}$ is needed to overcome the logarithmic loss in Lemma \ref{lem:17}.  This is a variance estimate for linear statistics of the GUE minor process.  As it turns out, the covariance estimates for eigenvalue gaps of GUE established by Cipolloni, Erd\H{o}s and Schr\"{o}der \cite{Cip}, combined with some further manipulations from the theory of determinantal processes to analyze the minor process, are \emph{almost} enough to obtain this sort of bound; there is however a technical difficulty because the bounds in \cite{Cip} are only established in the bulk of the spectrum and not on the edge.  However, the contributions coming from the edge region can be shown by relatively crude estimates to be (very slightly) non-trivial, and after removing these contributions to focus on the bulk contribution we will be able to make the above strategy work. In principle one could obtain stronger quantitative estimates than the $o(n^4)$ bound in Theorem \ref{mainthm} either by extending the results in \cite{Cip} to apply closer to the edge, or by making further progress on the KLS conjecture, but we will not attempt to do so here.

\subsection{Acknowledgments}
The first author is supported by a Swarna Jayanti fellowship and a grant associated with the Infosys-Chandrasekharan virtual center for Random Geometry. The second author is supported by NSF grants DMS-1712862 and DMS-2153742.
The third author is supported by NSF grant DMS-1764034 and by a Simons Investigator Award.  
We are  grateful to Amol Aggarwal for many helpful discussions, in particular for informing us of \cite{Cip}. We are grateful to Ronen Eldan for a helpful communication.

\section{Poincar\'e inequalities on log-concave measures}\label{log-concave-sec}

In this section we establish a useful Poincar\'e inequality over log-concave measures, which among other things implies Lemma \ref{lem:17}.  Namely, we show

\begin{proposition}[Poincar\'e inequality on log-concave measures]\label{log-concave-poin}
Let $\eta$ be an log-concave probability measure in $\R^d$ with finite second moments, and define the $d \times d$ inertia matrix $M$ by the formula
$$ M \coloneqq \E_\eta xx^T - (\E_\eta x) (\E_\eta x)^T.$$
Then for any Lipschitz function $f \colon \R^d \to \R$, one has
$$ \Var_\eta f \ll \left( \E_\eta |\nabla f|^2  \right)\| M \|_{\mathrm{op}} \log(2+d).$$
\end{proposition}

\begin{proof}  By a theorem of Borell \cite{borell}, log-concave probability measures are absolutely continuous with respect to some affine subspace of $\R^d$, so without loss of generality we can take $\eta$ to be absolutely continuous with some density $\rho$, so in particular the inertia matrix is non-singular. By translating we may assume the mean zero condition $\E_\eta x = 0$, and then by pushing forward by $M^{-1/2}$ we may assume that $M$ is the identity, that is to say we may assume that $\eta$ is an \emph{isotropic measure} in the sense that
$$ \E_\eta x = 0; \quad \E_\eta xx^T = I_d.$$

Define the \emph{Cheeger constant} $D_{\mathrm{Che}}(\eta)$ of $\eta$ (with respect to the Euclidean inner product) by the formula
$$ D_{\mathrm{Che}}(\eta) \coloneqq \inf_{A \subset \R^d} \frac{\int_{\partial A} \rho}{\min(\eta(A), 1 - \eta(A))}$$
where the infimum runs over all open subsets $A$ of $\R^d$ with smooth boundary with $0 < \eta(A) < 1$, and $\partial_A$ is integrated using surface measure.  By the Cheeger inequality in the form of \cite[Theorem 1.5]{Emanuel}, one has the Poincar\'e inequality
$$ D_{\mathrm{Che}}(\eta)^2 \Var_\eta f \ll \E_\eta |\nabla f|^2 $$
so the task reduces to (and is in fact equivalent to) the lower bound
$$D_{\mathrm{Che}}(\eta) \gg \frac{1}{\sqrt{\log(2+d)}}$$
on the Cheeger constant of an isotropic log-concave measure.  But this follows from
recent work of Klartag \cite[Theorem 1.2]{Klartag-KLS} (building upon previous advances in \cite{chen, KlartagLehec, Lee-Vemp}).
\end{proof}

\begin{remark} The KLS conjecture asserts the uniform lower bound $D_{\mathrm{Che}}(\eta) \gg 1$ for all isotropic log-concave measures $\eta$.  If true, this conjecture would allow us to remove the $\log(2+d)$ factor in the above proposition.
\end{remark}

\begin{proof}[Proof of Lemma \ref{lem:17}]  As in the proof of Proposition \ref{log-concave-poin}, we may assume that $\eta$ is isotropic; we may also normalize $\sup_{w \in {\mathcal W}} \Var_\eta(w) = 1$.  From the isotropy of $\eta$ we have
$$ |\nabla w|^2 = \Var_\eta(w) $$
for any affine form $w$.  In particular, each $w$ in ${\mathcal W}$ is $1$-Lipschitz, so $\sup_{w \in {\mathcal W}} w$ is also. The claim now follows from Proposition \ref{log-concave-poin}.
\end{proof}

The Poincar\'e inequality treats each of the $d$ basis vectors $e_1,\dots,e_d$ of $\R^d$ equally.  For our applications it is convenient to work with a non-isotropic version of this inequality in which different groups of basis vectors are treated with a different weight.

\begin{proposition}[Weighted Poincar\'e inequality on log-concave measures]\label{log-concave-poin-weight}
Let $\eta$ be an log-concave probability measure in $\R^d$ with finite second moments.  Express $\R^d$ as a Cartesian product $\R^{d_1} \times \dots \times \R^{d_k}$ for some $d_1,\dots,d_k$ summing to $d$ (so that a vector $x \in \R^d$ is expressed as $(x_1,\dots,x_k)$ for $x_j \in \R^{d_j}$), and for each $i=1,\dots,k$, and define the $d_j \times d_j$ inertia matrix $M_j$ by the formula
$$ M_j \coloneqq \E_\eta x_jx_j^T - (\E_\eta x_j) (\E_\eta x_j)^T.$$
Then for any Lipschitz function $f \colon \R^d \to \R$, and any weights $\alpha_1,\dots,\alpha_k > 0$, one has
$$ \Var_\eta f \ll \left( \E_\eta \sum_{j=1}^k \alpha_j |\nabla_j f|^2\right) \left( \sum_{j=1}^k \alpha_j^{-2} \| M_j \|_{\mathrm{op}}^2 \right)^{1/2} \log(2+d)$$
where $\nabla_j f \colon \R^d \to \R^{d_j}$ denotes the gradient with respect to the basis vectors of the $\R^{d_j}$ factor of $\R^d$.
\end{proposition}

\begin{proof}  By pushing forward $\eta$ by the map $(x_1,\dots,x_k) \mapsto (\alpha_1^{1/2} x_1, \dots, \alpha_k^{1/2} x_k)$ we may normalize $\alpha_j=1$ for all $j$, so that $\sum_{j=1}^k \alpha_j |\nabla_j f|^2 = |\nabla f|^2$.  By Proposition \ref{log-concave-poin-weight}, it thus suffices to establish the bound
$$ \| M \|_{\mathrm{op}} \leq \left( \sum_{j=1}^k \| M_j \|_{\mathrm{op}}^2 \right)^{1/2}.$$
For any $x = (x_1,\dots,x_k) \in \R^d$, where $x_j \in \R^{d_j}$ for all $j$, it follows from the positive semi-definiteness of $M$ and the triangle inequality followed by Cauchy--Schwarz that
\begin{align*}
 (x^T M x)^{1/2} &\leq \sum_{j=1}^k (x_j^T M_j x_j)^{1/2} \\
&\leq \sum_{j=1}^k \| M_j \|_{\mathrm{op}} |x_j| \\
&\leq \left( \sum_{j=1}^k \| M_j \|_{\mathrm{op}}^2\right)^{1/2} |x|
\end{align*}
and the claim follows.
\end{proof}

\section{The octahedron recurrence}\label{recurrence-sec}

Let $\lambda, \mu, \gamma, \pi \in \Spec^\circ$.  In \cite{KT2}, a volume-preserving, piecewise-linear \emph{octahedron recurrence}
\begin{equation}\label{octahedron-recur}
 \oct \colon \bigcup_\sigma \HIVE_{\sigma \boxplus \mu \rel \pi} \times \HIVE_{\gamma \boxplus \lambda \rel \sigma} \to \bigcup_\nu \HIVE_{\lambda \boxplus \mu \rel \nu} \times \HIVE_{\gamma \boxplus \nu \rel \pi}
\end{equation}
was constructed, and an explicit formula for it given using the work of Speyer \cite{Speyer}.  We now recall a version of this formula that will be convenient for our purposes.  We will identify a pair $(h,h') \in \HIVE_{\lambda \boxplus \mu \rel \nu} \times \HIVE_{\gamma \boxplus \nu \rel \pi}$ with a single function $\tilde h \colon T \cup T' \to \R$ defined on the square $T \cup T' = \{0,\dots,n\}^2$ by the formula
\begin{equation}\label{th-1}
 \tilde h(i,j) \coloneqq h(i,j)
\end{equation}
when $(i,j)$ lies in the triangle $T \coloneqq \{(i,j): 0 \leq i \leq j \leq n \}$, and
\begin{equation}\label{th-2}
\tilde h(i,j) \coloneqq h'(j,n-i+j) - \sum \gamma
\end{equation}
when $(i,j)$ lies in the opposite triangle $T' \coloneqq \{ (i,j): 0 \leq j \leq i \leq n \}$.  Note that both definitions agree on the diagonal $\{ (i,i): 0 \leq i \leq n \}$ due to the boundary values of the hives $h,h'$; see Figure \ref{fig:octahedron}.  The function $\tilde h$ will be rhombus concave on $T$ and on $T'$, but not necessarily concave along rhombi that cross the diagonal separating $T$ and $T'$.

\begin{figure}
\begin{center}
\includegraphics[scale=0.40]{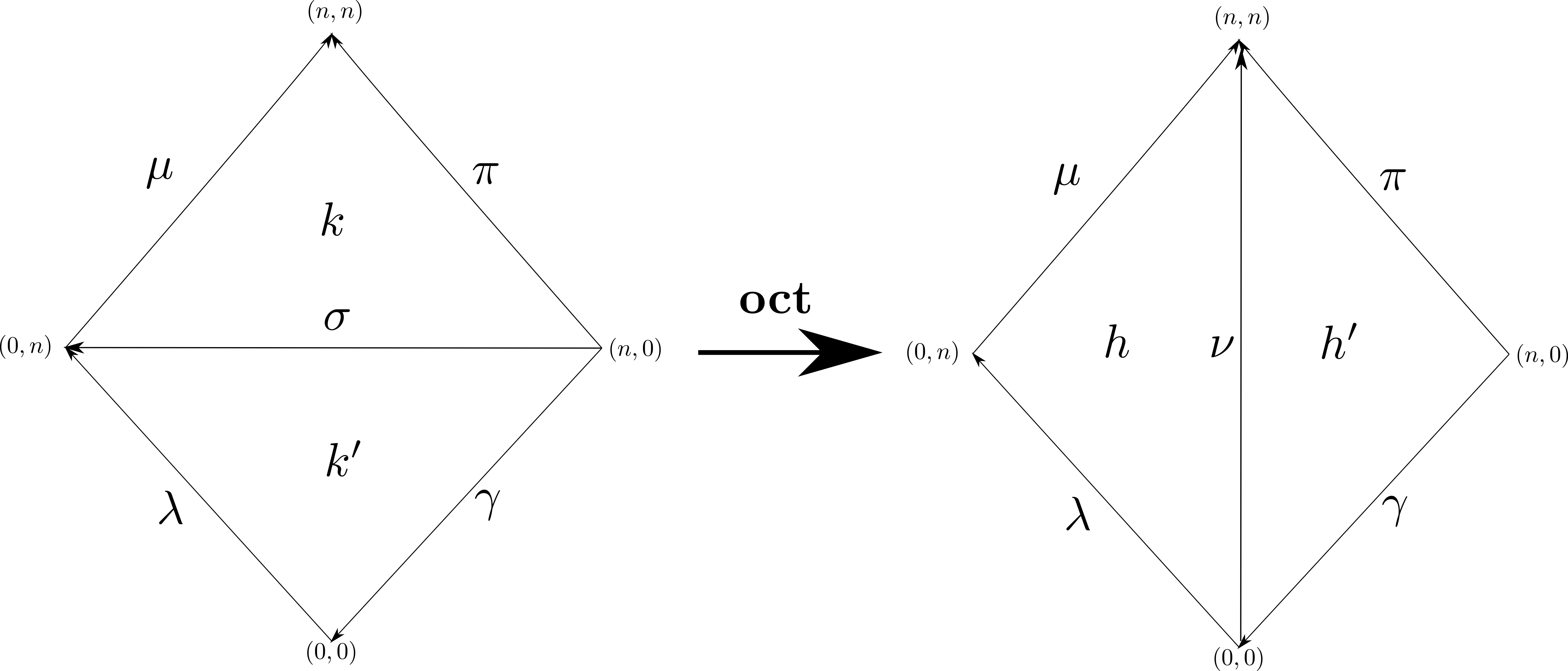}
\caption{A schematic depiction of the octahedron recurrence that transforms one pair $(k,k') \in \HIVE_{\sigma \boxplus \mu \rel \pi} \times \HIVE_{\gamma \boxplus \lambda \rel \sigma}$ of hives into another $(h,h') \in \HIVE_{\lambda \boxplus \mu \rel \nu} \times \HIVE_{\gamma \boxplus \nu \rel \pi}$.  The hives $h,h',k,k'$ have been shifted to lie on triangles $T, T', U, U'$ respectively. }\label{fig:octahedron}
\end{center}
\end{figure}

\begin{figure}
\begin{center}
\includegraphics[scale=0.40]{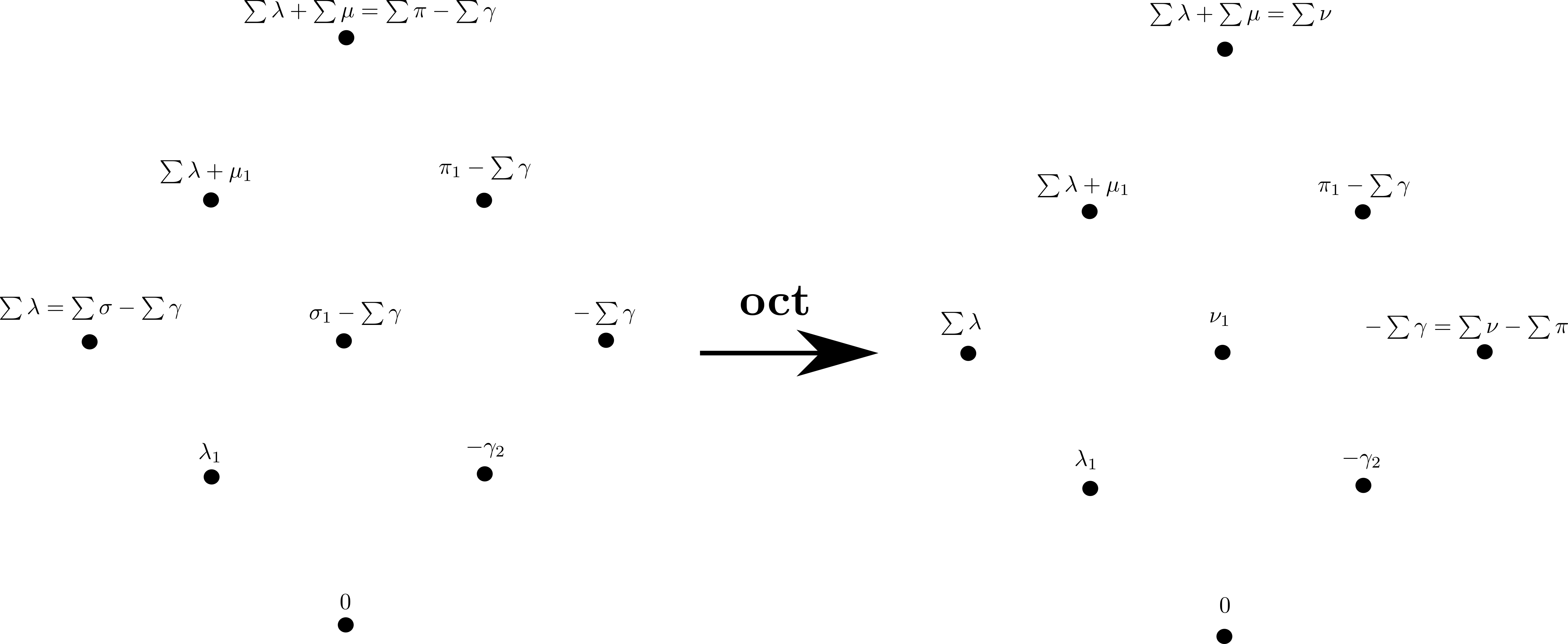}
\caption{The $n=2$ case of the octahedron recurrence.  One can determine the value $\nu_1$ in the right image from the data in the left image by the formula $\nu_1 = \max( \sum \lambda + \mu_1 + \gamma_1 - \sigma_1, \la_1 + \pi_1 - \sigma_1)$.}\label{fig:2d-octahedron}
\end{center}
\end{figure}

In a similar vein, we identify a pair $(k,k') \in \HIVE_{\sigma \boxplus \mu \rel \pi} \times \HIVE_{\gamma \boxplus \lambda \rel \sigma}$ with a single function $\tilde k \colon U \cup U' \to \R$ defined on the square\footnote{Even though $T \cup T'$ and $U \cup U'$ are both technically equal to the same set $\{0,\dots,n\}^2$, it is conceptually better to think of these sets as being distinct (except on the boundary).  In Appendix \ref{octahedron-app} we will view these two copies of $\{0,\dots,n\}^2$ as the upper and lower faces respectively of a certain tetrahedron $\mathtt{tet}$, and the octahedron recurrence $\oct$ can be constructed by ``excavating'' that tetrahedron.} $U \cup U' = \{0,\dots,n\}^2$ by the formula
\begin{equation}\label{tk-1}
 \tilde k(i,j) \coloneqq k(i+j-n, j) - \sum \gamma 
\end{equation}
when $(i,j)$ lies in the upper triangle $U \coloneqq \{ (i,j): i, j \leq n \leq i+j \}$, and
\begin{equation}\label{tk-2}
 \tilde k(i,j) \coloneqq k'(j,n-i) - \sum \gamma 
\end{equation}
when $(i,j)$ lies in the lower triangle $U' \coloneqq \{ (i,j): i, j \leq i+j \leq n \}$.  See Figure \ref{fig:octahedron}.  Again, the boundary values of $k,k'$ ensure that these definitions agree on the diagonal $\{ (i,n-i): 0 \leq i \leq n \}$.  

Given a pair $(k,k') \in \HIVE_{\sigma \boxplus \mu \rel \pi} \times \HIVE_{\gamma \boxplus \lambda \rel \sigma}$, the octahedron recurrence produces a pair $\oct(k,k') = (h,h')$ in $\HIVE_{\lambda \boxplus \mu \rel \nu} \times \HIVE_{\gamma \boxplus \nu \rel \pi}$ for some $\nu$.  To describe this recurrence, we will describe how the combined function $\tilde h$ defined by \eqref{th-1}, \eqref{th-2} depends on the combined function $\tilde k$ defined by \eqref{tk-1}, \eqref{tk-2}.  For vertices $v = (i,j)$ on the boundary of $\{0,\dots,n\}^2$ (i.e., $i \in \{0,n\}$ or $j \in \{0,n\}$), the octahedron recurrence does nothing:
$$ \tilde h(i,j) \coloneqq \tilde k(i,j).$$
Note that this is consistent with the boundary conditions placed on $h,h',k,k'$.  

For vertices $v = (i,j)$ in the interior or $\{0,\dots,n\}^2$, the octahedron recurrence specifying $\tilde h(i,j)$ is more complicated to describe.  It was initially defined by recursively ``excavating'' a real-valued function on a tetrahedron $\{ (a,b,c,d) \in \Z^4: a,b,c,d \geq 0; a+b+c+d = n \}$ with $\tilde k$ describing the values on the top two faces, and $\tilde h$ the bottom two faces; see Appendix \ref{octahedron-app}.  An alternate description was given by Speyer \cite{Speyer} (and reproduced in \cite[\S 7]{KT2}), in terms of perfect matchings of an ``excavation graph'' associated to $(i,j)$.  We will use a modification of Speyer's formula that is more convenient for our purposes, in which the perfect matchings are replaced the dual concept of a \emph{lozenge tiling}.  To describe this formula we need some definitions.

\begin{definition}[Lozenges and border triangles]  A \emph{lozenge} is a quadruple $ABCD$ in $U$ or $U'$ that is one of following three forms for some $i,j \in \Z$:
\begin{itemize}
\item[(i)] $(A,B,C,D) = ((i,j), (i+1,j-1), (i+2,j-1), (i+1,j))$
\item[(ii)] $(A,B,C,D) = ((i,j), (i,j+1), (i-1,j+2), (i-1,j+1))$
\item[(iii)] $(A,B,C,D) = ((i,j), (i+1,j), (i+1,j+1), (i,j+1))$.
\end{itemize}
Lozenges of type (i) will be called \emph{blue} if they lie in $U$ and \emph{red} if they lie in $U'$; lozenges of type (ii) will be called \emph{red} if they lie in $U$ and \emph{blue} if they lie in $U'$; and lozenges of type (iii) that lie either in $U$ or in $U'$ will be called \emph{green}; see Figure \ref{fig:typical}.  A quadruple of the form (iii) that crosses the diagonal separating $U$ and $U'$ is \emph{not} considered to be a lozenge, but instead splits into two border triangles as defined below.  (The colors of lozenges will not be needed immediately, but will play a useful role later in this section.)

A \emph{border edge} is an edge $AC$ of the form $(A,C) = ((i,n-i), (i+1,n-i-1))$ for some $0 \leq i < n$; the border edges thus separate $U$ and $U'$.  Each border edge $(A,C) = ((i,n-i), (i+1,n-i-1))$ is bordered by two \emph{border triangles} $ABC$, defined as follows:
\begin{itemize}
\item (Upward triangle) $(A,B,C) = ((i,n-i), (i+1,n-i), (i+1,n-i-1))$.
\item (Downward triangle) $(A,B,C) = ((i,n-i), (i,n-i-1), (i+1,n-i-1))$.
\end{itemize}
Note that upward triangles lie (barely) in $U$, while downward triangles lie (barely) in $U'$.

Given a lozenge $\edge = ABCD$ and a function $\tilde k \colon \{0,\dots,n\}^2 \to \R$ defined as before, we define the \emph{weight} $\weight(\edge) = \weight(\edge, \tilde k)$ to be the quantity
$$\weight(\edge) \coloneqq \frac{1}{3} (\tilde k(A) + \tilde k(C) - \tilde k(B) - \tilde k(D)).$$
Similarly, given a border triangle $\Delta = ABC$, the weight $\weight(\Delta) = \weight(\tau, \tilde k)$ is defined as\footnote{One could also define the weight here more symmetrically as $\frac{1}{6} \tilde k(A) - \frac{1}{3} \tilde k(B) + \frac{1}{6} \tilde k(C)$, provided that one also adjusts the formula \eqref{hexagon-weight} below to $\frac{1}{3}(\tilde k(B) + \tilde k(C) + \tilde k(E) + \tilde k(F)) - \frac{1}{6} (\tilde k(A) + \tilde k(D))$; however the asymmetric form is more convenient technically for our application.}
$$\weight(\Delta) \coloneqq \frac{1}{3} (\tilde k(A) - \tilde k(B)).$$
\end{definition}

\begin{definition}[Octahedron recurrence]\label{octa-def}  If $v = (i,j)$ lies in the interior of $\{0,\dots,n\}^2 = T \cup T'$, then the \emph{excavation hexagon}
$\hexagon_v = ABCDEF$ in $\{0,\dots,n\}^2 = U \cup U'$ centered at $v$ is defined as follows:
\begin{itemize}
\item If $v \in T$ (i.e., $i \leq j$), then
$$ (A,B,C,D,E,F) = ((0,n), (0,j), (i,j-i), (n+i-j,j-i), (n+i-j,j), (i,n)).$$
\item If $v \in T'$ (i.e., $i \geq j$), then
$$ (A,B,C,D,E,F) = ((i-j, n+j-i), (i-j,j), (i,0), (n,0), (n,j), (i,n+j-i)).$$
\end{itemize}
Note that these two definitions agree when $v \in T \cap T'$ (i.e., when $i=j$). The original point $v = (i,j)$ is then the intersection of the diagonals $BE$ and $CF$.  The line $AD$ is called the \emph{equator}; it lies on the border between $U$ and $U'$.  The \emph{weight} $\weight(\hexagon_v) = \weight(\hexagon_v,\tilde k)$ of this hexagon is defined as
\begin{equation}\label{hexagon-weight}
 \weight(\hexagon_v) \coloneqq \frac{1}{3} (\tilde k(B) + \tilde k(C) - \tilde k(D) + \tilde k(E) + \tilde k(F)).
\end{equation}
A \emph{lozenge tiling}\footnote{More precisely, this is a ``lozenge and border triangle'' tiling.} $\Xi$ of the excavation hexagon $\hexagon_v$ is a partition of the (solid) hexagon into (solid) lozenges and (solid) border triangles, such that each border edge on the equator is adjacent to exactly one border triangle in the tiling; see Figure \ref{fig:typical}.  An example of a lozenge tiling is the \emph{standard lozenge tiling} $\Xi_0$, in which the trapezoid $ABEF$ is tiled by blue lozenges in $U$ and by green lozenges and downward border triangles in $U'$, while the opposite trapezoid $BCDE$ is tiled by green lozenges and upward border triangles in $U$ and by red lozenges in $U'$; see Figure \ref{fig:standard}.

The \emph{weight} $w_\Xi = w_\Xi(\tilde k)$ of such a tiling is defined to be the sum of the weights of all the lozenges $\edge$ and triangles $\Delta$ in the tiling, as well as the weight of the entire hexagon $\hexagon_v$:
\begin{equation}\label{weight-form}
w_\Xi \coloneqq \sum_{\edge \in \Xi} \weight(\edge) + \sum_{\Delta \in \Xi} \weight(\Delta) + \weight(\hexagon_v).
\end{equation}
Note that the $w_\Xi$ depend linearly on $\tilde k$, and hence on $k, k'$.  We then define
$$ \tilde h(v) \coloneqq \max_{\Xi\ \mathrm{tiles}\ \hexagon_v} w_\Xi$$
\end{definition}

\begin{example} As a simple example, take $v = (1, n-1) \in T$ (assuming $n \geq 2$), then the excavation hexagon $\hexagon_v = ABCDEF$ is given by the unit hexagon centered at $v$:
$$ (A,B,C,D,E,F) = ((0,n), (0,n-1), (1,n-2), (2,n-2), (2, n-1), (1,n)).$$
This hexagon has two lozenge tilings, depicted in Figure \ref{fig:recurrence-example}.  For the tiling on the left, the blue lozenge has weight
$$ \frac{1}{3} (\tilde k(0,n) + \tilde k(2,n-1) - \tilde k(1,n) - \tilde k(1,n-1))$$
the red lozenge has weight
$$ \frac{1}{3} (\tilde k(0,n-1) + \tilde k(2,n-2) - \tilde k(1,n-1) - \tilde k(1,n-2))$$
the upward triangle has weight
$$ \frac{1}{3} (\tilde k(2,n-1) - \tilde k(1,n-1))$$
the downward triangle has weight
$$ \frac{1}{3} (\tilde k(0,n-1) - \tilde k(0,n))$$
and the hexagon has weight
$$ \frac{1}{3} (\tilde k(0,n-1) + \tilde k(1,n-2) - \tilde k(2,n-2) + \tilde k(2,n-1) + \tilde k(1,n)$$
leading to a total weight of
$$ \tilde k(2,n-1) + \tilde k(0,n-1) - \tilde k(1,n-1).$$
The tiling on the right can similarly be computed to have a total weight of
$$ \tilde k(1,n) + \tilde k(1,n-2) - \tilde k(1,n-1)$$
leading to the familiar octahedron relation
$$ \tilde h(1,n-1) = \max( \tilde k(2,n-1) + \tilde k(0,n-1) , \tilde k(1,n) + \tilde k(1,n-2) ) - \tilde k(1,n-1).$$
\end{example}

\begin{example} The lozenge tiling in Figure \ref{fig:typical} has weight
$$\tilde k(4,4) + \tilde k(3,3) - \tilde k(4,2) - \tilde k(2,3) + \tilde k(1,3) - \tilde k(2,2)  -\tilde k(3,1) + \tilde k(4,0) + \tilde k(2,1).$$
\end{example}

\begin{figure}
\begin{center}
\includegraphics[scale=0.40]{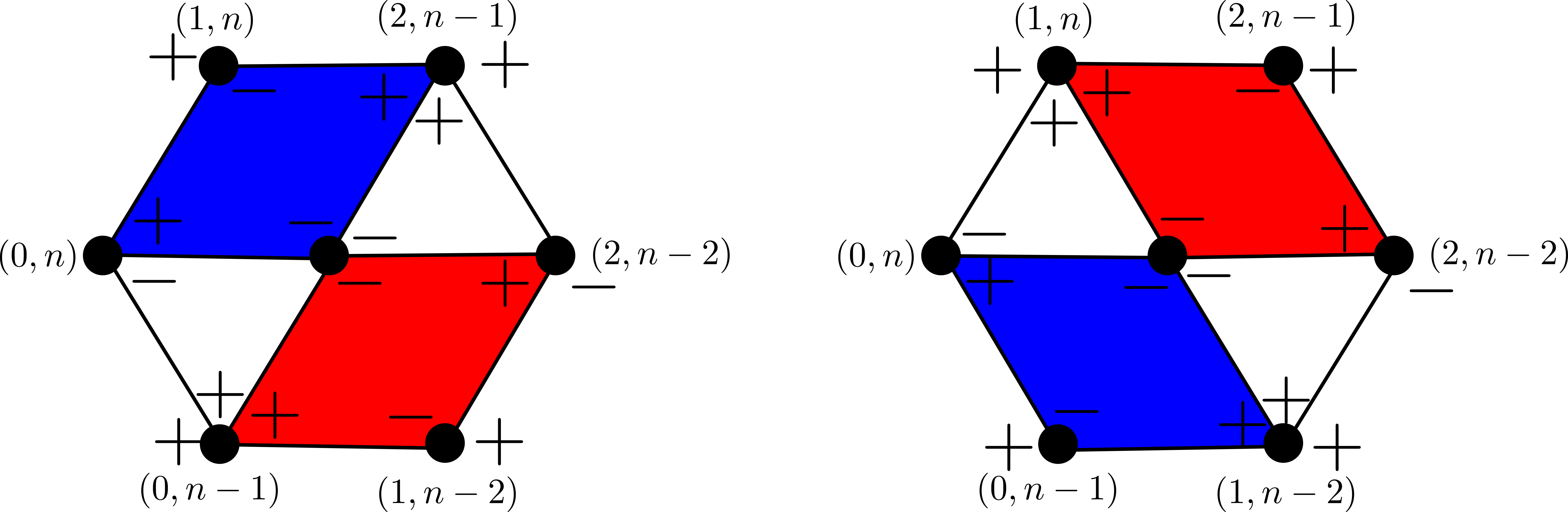}
\caption{The two lozenge tilings of $\hexagon_{(1,n-1)}$. The weight coefficients of the lozenges, border triangles, and hexagon are marked with $+$ (for a weight of $+1/3$) and $-$ (for a weight of $-1/3$). The tiling on the left is the standard one.}\label{fig:recurrence-example}
\end{center}
\end{figure}

\begin{figure}
\begin{center}
\includegraphics[scale=0.40]{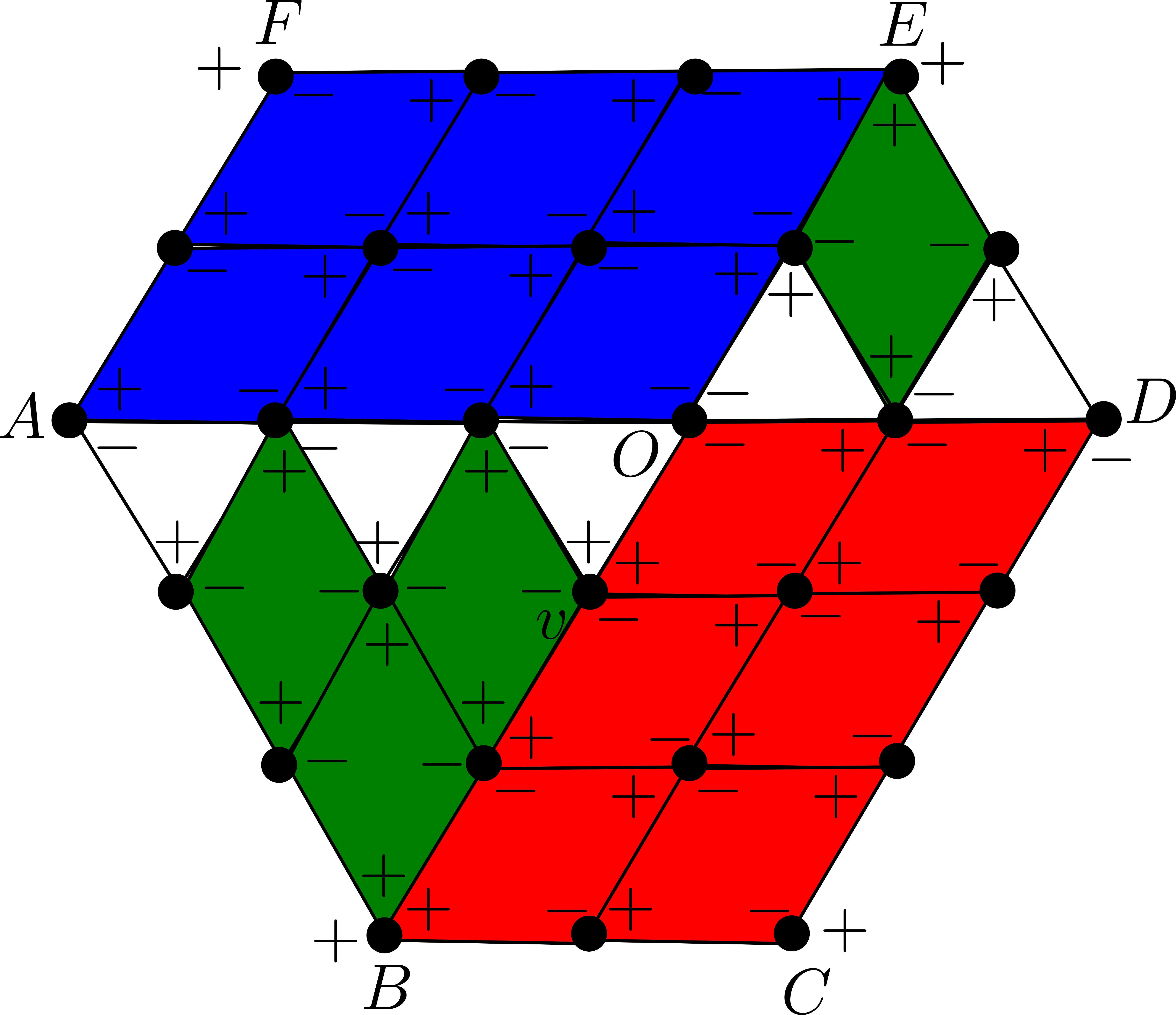}
\caption{The standard lozenge tiling of a hexagon $ABCDEF$ centered at $v$.  The total weight of this tiling is $\tilde k(E) + \tilde k(B) - \tilde k(O)$, where $O$ is the intersection of the diagonal $BE$ with the equator $AD$.}\label{fig:standard}
\end{center}
\end{figure}

\begin{figure}
\begin{center}
\includegraphics[scale=0.40]{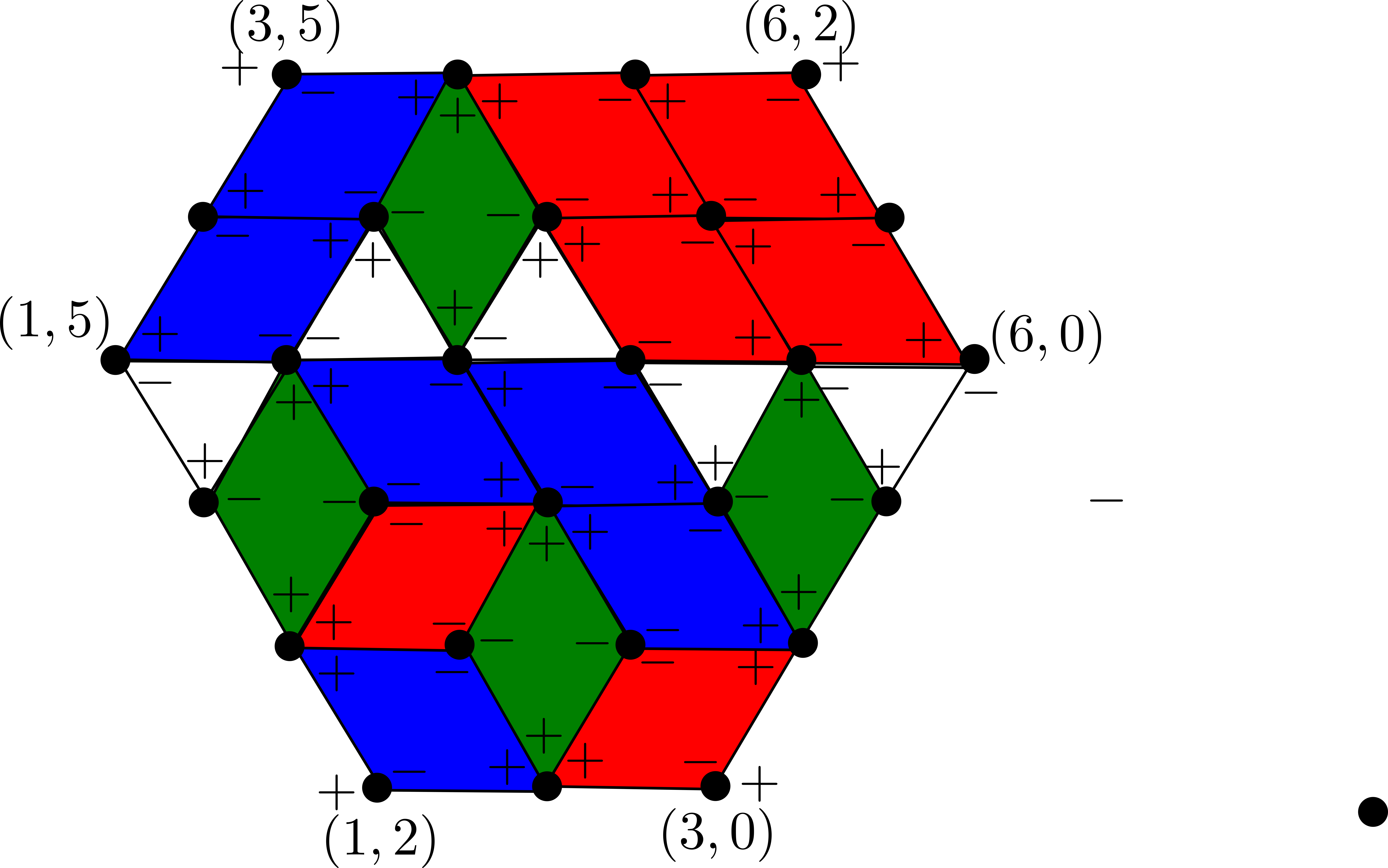}
\caption{A typical lozenge tiling of $\hexagon_{(2,3)}$, $n=6$. }\label{fig:typical}
\end{center}
\end{figure}

\begin{theorem}\label{octahedron}  The construction above agrees with the octahedron recurrence described in \cite{KT2}.  
\end{theorem}

\begin{proof}  This is basically a matter of comparing notations with the Speyer formula in \cite[\S 7]{KT2} and performing some routine calculations; we do so in Appendix \ref{octahedron-app}.
\end{proof}

As a consequence of this theorem and the results in \cite{KT2}, the octahedron recurrence \eqref{octahedron-recur} that we have just defined is indeed a volume-preserving\footnote{In \cite{KT2} the stronger assertion that this recurrence is a bijection between integer lattice points is established, but the volume-preserving nature of the bijection then follows by a standard rescaling and limiting argument to pass from the discrete to the continuous setting.} bijection between the polytopes in \eqref{octahedron-recur}.
 
In our analysis, the contribution of the red lozenges is inconvenient to work with directly (these lozenges cross the ``creases'' in Figure \ref{fig:excavated}, and are thus sensitive to the choice of $\gamma$).  However, it turns out that they can essentially be replaced with the contribution of the blue lozenges:

\begin{lemma}[Replacing red lozenges with blue]  Let $v$ be an interior point of $\{0,\dots,n\}^2$, and let $\Xi$ be a lozenge tiling of $\hexagon_v = ABCDEF$.  Then one has the identity
\begin{equation}\label{xio}
\sum_{\edge \in \Xi \text{, red}} \weight(\edge) - \sum_{\edge \in \Xi \text{, blue}} \weight(\edge)
= \frac{1}{3} (-\tilde k(A) + \tilde k(B) - \tilde k(C) + \tilde k(D) - \tilde k(E) + \tilde k(F)).
\end{equation}
\end{lemma}

\begin{proof}  The following argument can be viewed as implementing a discrete analogue of Stokes' theorem.
Define an auxiliary weight $\weight^*(PQR)$ on unit triangles $PQR$ in $U$ or $U'$ that are either upward pointing
$$ (P,Q,R) = ((i,j), (i+1,j), (i+1,j-1))$$
or downward pointing
$$ (P,Q,R) = ((i,j), (i,j-1), (i+1,j-1))$$
by defining 
$$\weight^*(PQR) \coloneqq \frac{1}{3} (\tilde k(R) - \tilde k(P))$$
for upward-pointing triangles in $U$ or downward-pointing triangles in $U'$, and
$$\weight^*(PQR) \coloneqq \frac{1}{3} (\tilde k(P) - \tilde k(R))$$
for downward-pointing triangles in $U$, or upward-pointing triangles in $U'$. One then observes:
\begin{itemize}
\item The sum of the auxiliary weights of the two component triangles of a red lozenge is equal to the weight of that lozenge.
\item The sum of the auxiliary weights of the two component triangles of a blue lozenge is equal to the negative of the weight of that lozenge.
\item The sum of the auxiliary weights of the two component triangles of a green lozenge is zero.
\item The sum of the auxiliary weights of all the boundary triangles in the tiling telescopes to $\frac{1}{3}(\tilde k(D) - \tilde k(A))$.
\end{itemize}
Summing, we conclude that the sum of the auxiliary weights of all the triangles in $\hexagon_v$ is equal to the left-hand side of \eqref{xio}
plus $\frac{1}{3}(\tilde k(D) - \tilde k(A))$.  On the other hand, this sum also telescopes to equal $\frac{1}{3}(\tilde k(F) - \tilde k(E) + 2\tilde k(D) - 2\tilde k(A) + \tilde k(B) - \tilde k(C))$.  The claim follows.
\end{proof}

In view of this lemma, we can also write the weight form $w_\Xi$ in \eqref{weight-form} in a red lozenge-avoiding form as
\begin{equation}\label{alt-weight}
w_\Xi \coloneqq 2 \sum_{\edge \in \Xi\text{, blue}} \weight(\edge) + \sum_{\edge \in \Xi\text{, green}} \weight(\edge) 
+ \sum_{\Delta \in \Xi} \weight(\Delta) + \weight'(\hexagon_v)
\end{equation}
where the modified weight $\weight'(\hexagon_v)$ of the hexagon $\hexagon_v$ is defined by the formula
$$
 \weight'(\hexagon_v) \coloneqq \frac{1}{3} (-\tilde k(A) + 2\tilde k(B) + 2\tilde k(F)).
$$
Now we specialize to the situation where $\gamma$ has large gaps, in the sense that
$$ \min_{1 \leq i < n-1} \gamma_i - \gamma_{i+1} > (\lambda_1-\lambda_n) + (\mu_1 - \mu_n).$$
It is a routine matter to use rhombus concavity to obtain
$$ \min_{1 \leq i < n-1} \mu_i - \mu_{i+1} \geq \min_{1 \leq i < n-1} \gamma_i - \gamma_{i+1} - (\lambda_1-\lambda_n) > \mu_1 - \mu_n$$
while from the Weyl identities one also has
$$ \min_{1 \leq i < n-1} \gamma_i - \gamma_{i+1} > \nu_1-\nu_n.$$
If one then sets $\sigma = \gamma + b$ and $\pi = \gamma + a$ for some $a,b \in \R^n$, then we conclude from Proposition \ref{gt-rem} that we have identifications
$$ 
\HIVE_{\gamma \boxplus \lambda \rel \sigma} \equiv \GT_{\diag(\lambda) \rel b}; \quad
\HIVE_{\sigma \boxplus \mu \rel \pi} \equiv \GT_{\diag(\mu) \rel a-b}; \quad \HIVE_{\gamma \boxplus \nu \rel \pi} \equiv \GT_{\diag(\nu) \rel a}.$$
Because of this, the octahedron recurrence can also be viewed (by abuse of notation) as a map
$$
 \oct \colon \bigcup_b \GT_{\diag(\la) \rel b} \times \GT_{\diag(\mu) \rel a-b} \to \bigcup_\nu \AHIVE_{\diag(\la \boxplus \mu \rel \nu) \rel a}
$$
as asserted in \eqref{oct-map}.  A priori, this map could depend on the choice of $\gamma$.  However, direct inspection of the definitions shows that the weight of a blue or green lozenge, or of a boundary triangle, does not depend on $\gamma$ (all the shifts by partial sums of $\gamma$ cancel each other out).  Furthermore, the modified weight $\weight'(\hexagon_v)$ of a hexagon $\hexagon_v$ can be verified to also be independent of $\gamma$ (and in fact only depends on the boundary data $\lambda, \mu$) when $v$ lies in $T$, and is equal to $-\sum_{k=n+i-j+1}^n \gamma_k$ plus a quantity independent of $\gamma$ when $v = (i,j)$ lies in $T$.  From this one can check that the map \eqref{oct-map} does not in fact depend on $\gamma$.

By the preceding discussion, this map $\oct$ is volume-preserving, and hence also extends by the Fubini--Tonelli theorem to a volume-preserving bijection
$$
 \oct \colon \GT_{\diag(\la) \rel \ast} \times \GT_{\diag(\mu) \rel \ast} \to \AHIVE_{\diag(\la \boxplus \mu \rel \ast) \rel \ast}
$$
for any $\la, \mu \in \Spec^\circ$.  By a further application of Fubini--Tonelli, the probability measure on $\AHIVE_{\diag(\ast \boxplus \ast \rel \ast) \rel \ast}$ whose density function equals \eqref{css} on the slices \eqref{ahive-slice}, is the pushforward under $\oct$ of the probability measure on $\GT_{\diag(\ast) \rel \ast} \times \GT_{\diag(\ast) \rel \ast}$ whose density function equals \eqref{css} on the slices
$\GT_{\diag(\la) \rel \ast} \times \GT_{\diag(\mu) \rel \ast}$.  By Proposition \ref{gt-rem}(iii), this latter probability distribution is the law of the pair $(g,g')$ of Gelfand--Tsetlin patterns generated by applying the minor process to random Hermitian matrices $A,B$ with 
$\frac{A}{\sqrt{\sigma_\lambda^2 n}}, \frac{B}{\sqrt{\sigma_\mu^2 n}}$ drawn independently from the GUE ensemble as per that proposition.  The proof of Theorem \ref{mainthm} is thus reduced to

\begin{proposition}[Reduction to the minor process]\label{w-minor}  Let $\sigma_\lambda, \sigma_\mu > 0$ be fixed, and let $A,B$ with 
$\frac{A}{\sqrt{\sigma_\lambda^2 n}}, \frac{B}{\sqrt{\sigma_\mu^2 n}}$ be drawn independently from the GUE ensemble, and let $(g,g')$ be the resulting Gelfand--Tsetlin patterns.  Then for any $v \in T$, we have
\begin{equation}\label{vare}
 \Var \max_{\Xi \text{ tiles } \hexagon_v} 2 \sum_{\edge \in \Xi\text{, blue}} \weight(\edge) + \sum_{\edge \in \Xi\text{, green}} \weight(\edge) 
+ \sum_{\Delta \in \Xi} \weight(\Delta) + \weight'(\hexagon_v) = o(n^4)
\end{equation}
where we identify $(g,g')$ with a pair of hives $(k,k')$ using a large gaps tuple $\gamma$ as indicated above (noting from \eqref{alt-weight} that the choice of $\gamma$ does not affect the value of $w_\Xi(g,g')$).
\end{proposition}

The proof of this proposition will occupy most of the remaining sections of the paper.

\section{Using eigenvalue rigidity to remove edge contributions}

Let the notation and hypotheses be as in Proposition \ref{w-minor}.  We now use known eigenvalue rigidity bounds for GUE matrices to remove some ``edge'' contributions to the linear forms $w_\Xi$ as being negligible.

We begin with a crude estimate.  If $\lambda$ denotes the eigenvalues of $A$, then we can easily obtain the bound
$$ \E \sum_{i=1}^4 \lambda_i^4 = \E \mathrm{tr} A^4 \ll n^{O(1)}$$
(in fact far more precise asymptotics are possible, but we will not need them here).  In particular $\E |\lambda|^4 \ll n^{O(1)}$.  Similarly, if $\mu$ denotes the eigenvalues of $B$, then $\E |\mu|^4 \ll n^{O(1)}$.  In particular, by the interlacing property, all the components of $g,g'$ have a fourth moment of polynomial size $O(n^{O(1)})$.  By H\"older's inequality, this has the following consequence: if $E$ is an event of superpolynomially small probability (i.e., $\P E = O(n^{-C})$ for any fixed $C$), then any reasonable statistic $X$ of $g,g'$ (such as a component of the random variable appearing in \eqref{vare}) will have negligible second moment on $E$, in the sense that $\E |X|^2 \one_E \ll n^{-C}$ for any fixed $C$ (where $\one_E$ of course denotes the indicator function of $E$).  As a consequence, we may freely remove such events from our analysis, and restrict attention to events that hold with overwhelming probability (probability $1-O(n^{-C})$ for any $C>0$).

We now recall an eigenvalue rigidity estimate:

\begin{lemma}[Eigenvalue rigidity]\lab{lem:TV}
  Let $A$ be a matrix with $A/\sqrt{n}$ having the distribution of GUE.  Then for
  any $1 \leq i \leq n$ we have
$$\p(n^{-1/3} \min(i, n - i + 1)^{1/3}|\la_i - \sqrt{n} \gamma_i| \geq T )\ll  n^{O(1)} exp(-cT^c)$$
for any $T > 0$ and some absolute constant $c>0$, where the \emph{classical location} $\gamma_i$ is the value predicted by the semicircular law:
$$ \int_{-\infty}^{\gamma_i} \frac{1}{2\pi} (4-x^2)^{1/2}\ dx = \frac{i}{n}.$$
In particular,
$$ \la_i,  \E \la_i = \sqrt{n} \gamma_i + O( n^{1/3} \min(i, n - i + 1)^{-1/3} \log^{O(1)} n)
$$
with overwhelming probability. 
\end{lemma}

\begin{proof} See for instance \cite[Corollary 15]{TaoVu-conc} (in which this statement is established for the broader class of Wigner random matrices).
\end{proof}

From Lemma \ref{lem:TV}, we conclude in particular that
$$ \lambda_n = \E \lambda_n + O(n^{1/3} \log^{O(1)} n) $$
with overwhelming probability.  Since the $k \times k$ minor of a GUE matrix is also a GUE matrix, we similarly have
$$ \lambda_{k,k} = \E \lambda_{k,k} + O(n^{1/3} \log^{O(1)} n)$$
with overwhelming probability for all $1 \leq k \leq n$, where we recall that $\lambda_{1,k} \geq \dots \geq \lambda_{k,k}$ are the eigenvalues of the top left $k \times k$ minor of $A$ (and thus form the $k^{\mathrm{th}}$ row of $g$).  Similarly we have
$$ \mu_{1,k} = \E \mu_{1,k} + O(n^{1/3} \log^{O(1)} n)$$
with overwhelming probability for all $1 \leq k \leq n$.  One could compute these expectations more precisely using the Tracy--Widom law \cite{tracy}, but we will not need to do so here.

Let $\Delta$ be a border triangle associated to a border edge $((i,n-i), (i+1,n-i-1))$ for some $0 \leq i < n$.  By inspecting the definitions, we see that the weight of this triangle is given by the formula
$$ \weight(\Delta) = \frac{2}{3} \mu_{1,n-i}$$
if $\Delta$ is an upward pointing triangle, and
$$ \weight(\Delta) = \frac{1}{3} \lambda_{n-i,n-i}$$
if $\Delta$ is a downward pointing triangle.  We conclude that for any lozenge tiling $\Xi$ of $\hexagon_v$, we have
$$ \sum_{\Delta \in \Xi} \weight(\Delta) = \sum_{\Delta \in \Xi} \E \weight(\Delta) + O( n^{4/3} \log^{O(1)} n )$$
with overwhelming probability.  
By the triangle inequality, the contribution of the error term $O( n^{4/3} \log^{O(1)} n )$ is negligible for the purposes of establishing, so without loss of generality we can replace $\weight(\Delta)$ with $\E \weight(\Delta)$ in \eqref{vare}.

The weight $\weight'(\hexagon_v)$ is a certain linear combination of the eigenvalues $\lambda_i, \mu_j$ with bounded coefficients.  By Lemma \ref{lem:TV}, we conclude that
\begin{align*}
 \weight'(\hexagon_v) &= \E \weight'(\hexagon_v) + O\left( \sum_{i=1}^n n^{1/3} \min(i, n - i + 1)^{-1/3} \log^{O(1)} n \right)\\
&= \E \weight'(\hexagon_v) + O( n \log^{O(1)} n) .
\end{align*}
Again, the contribution of the $O( n \log^{O(1)} n )$ error is acceptable, so we may also replace $\weight'(\hexagon_v)$ by $\E \weight'(\hexagon_v)$ in \eqref{vare}.

It remains to control the contribution of the weights of the blue and green lozenges.  In the upper region $U$, a blue lozenge with vertices
$$ ((i,j), (i+1,j-1), (i+2,j-1), (i+1,j))$$
can be computed to have weight
$$ \frac{1}{3} ( \mu_{i+j+1-n,j-1} - \mu_{i+j+1-n,j} )$$
while a green lozenge with vertices
$$ ((i,j), (i+1,j), (i+1,j+1), (i,j+1))$$
can be computed to have weight
$$ \frac{1}{3} ( \mu_{i+j+2-n,j+1} - \mu_{i+j+1-n,j} ).$$
Similarly, in the lower region $U'$, a blue lozenge with vertices
$$ ((i,j), (i,j+1), (i-1,j+2), (i-1,j+1))$$
can be computed to have weight
\begin{equation}\label{j3}
 \frac{1}{3} ( \lambda_{j+2,n-i+1} - \lambda_{j+1,n-i} )
\end{equation}
while a green lozenge with vertices
$$ ((i,j), (i+1,j), (i+1,j+1), (i,j+1))$$
can be computed to have weight
$$ \frac{1}{3} ( \lambda_{j+1,n-i} - \lambda_{j+1,n-i+1} ).$$
In particular, by the interlacing inequalities \eqref{interlacing}, all these weights are non-positive.

Let $\eps>0$ be a small parameter, and let $U_\eps$ denote the portion of $U$ that lies at Euclidean distance at least $\eps n$ from the boundary of $U$.   Define $U'_\eps$ similarly.  We now claim the estimate
$$ \sum_{\edge \not \subset U_\eps \cup U'_\eps} |\weight(\edge)| \ll \eps^{1/3} n^2$$
with overwhelming probability, where the sum is over all blue or green lozenges in $U$ or $U'$ that are not contained in $U_\eps$ or $U'_\eps$.  Indeed, by the preceding calculations, the preceding sum telescopes to be bounded by a sum 
\begin{itemize}
\item[(i)] $O(\eps n)$ expressions of the form $\lambda_{k,1} - \lambda_{k,k}$ or $\mu_{k,1} - \mu_{k,k}$ for various $1 \leq k \leq n$;
\item[(ii)] $O(n)$ expressions of the form $\lambda_{k,1} - \lambda_{k,i}$ or $\mu_{k,1} - \mu_{k,i}$ for various $1 \leq i \leq k \leq n$ with $i = O(\eps n)$; and
\item[(iii)] $O(n)$ expressions of the form $\lambda_{k,i} - \lambda_{k,k}$ or $\mu_{k,i} - \mu_{k,k}$ for various $1 \leq i \leq k \leq n$ with $k-i = O(\eps n)$.
\end{itemize}
By eigenvalue rigidity (Lemma \ref{lem:TV}), all the expressions in (i) are of size $O(n)$ with overwhelming probability, while all the expressions in (ii), (iii) are of size $O( \eps^{1/3} n )$ with overwhelming probability.  The claim follows.

By the triangle inequality, and by sending $\eps$ slowly to zero, it now suffices to establish the bound
$$
 \Var \max_{\Xi \text{ tiles } \hexagon_v} 2 \sum_{\edge \in \Xi\text{, blue}: \edge \subset U_\eps \cup U'_\eps } \weight(\edge) + \sum_{\edge \in \Xi\text{, green}: \edge \subset U_\eps \cup U'_\eps} \weight(\edge) 
+ \sum_{\Delta \in \Xi} \E \weight(\Delta) + \E \weight'(\hexagon_v) = o(n^4)
$$
as $n \to \infty$ for any fixed $\eps>0$.  Applying Lemma \ref{lem:17}, and noting that the deterministic quantities
$\E \weight(\Delta)$, $\E \weight'(\hexagon_v)$ have zero variance, it will thus suffice to establish the bound
$$
 \Var\left( 2 \sum_{\edge \in \Xi\text{, blue}: \edge \subset U_\eps \cup U'_\eps} \weight(\edge) + \sum_{\edge \in \Xi\text{, green}: \edge \subset U_\eps \cup U'_\eps} \weight(\edge)\right) = O(n^{4-c+o(1)})
$$
for each individual lozenge tiling $\Xi$ and some constant $c>0$, assuming $n$ is sufficiently large depending on $\eps$, and where we now allow implied constants in the $O()$ notation to depend on $\eps$.  

Henceforth we fix $\eps>0$ and assume $n$ sufficiently large depending on $\eps$. By the triangle inequality, it thus suffices to establish the bound
\begin{equation}\label{var-blue}
 \Var \sum_{\edge \in \Xi\text{, blue}: \edge \subset U'_\eps} \weight(\edge) = O(n^{4-c+o(1)})
\end{equation}
and similarly with blue replaced by green, or $U'_\eps$ replaced by $U_\eps$, or both.

We focus on establishing \eqref{var-blue}, as the other three cases are proven similarly.  By \eqref{j3}, it suffices to establish the bound
$$ \Var \sum_{(j,k) \in \Omega} \lambda_{j,k+1} - \lambda_{j,k} = O(n^{4-c+o(1)})$$
whenever $\Omega$ is a collection of tuples of integers $1 \leq j \leq k \leq n$ with $j, k-j, n-k \gg \eps n$.  By the triangle inequality, it suffices to show that
$$ \Var \sum_{j \in S_k} \lambda_{j,k+1} - \lambda_{j,k} = O(n^{2-c+o(1)})$$
for each $\eps n \ll k \leq n-1$, where $S_k$ is some subset of the bulk region $\{ 1 \leq j \leq k: j, k-j \gg \eps n \}$.  Since the minor of a GUE matrix is a rescaled version of a GUE matrix, it suffices to establish this claim for the case $k=n-1$, that is to say (after adjusting $\eps$ slightly) to show that
$$ \Var X_S = O(n^{2-c+o(1)})$$
for an arbitrary subset $S$ of $\{ 2\eps n \leq j \leq (1-2\eps) n\}$,
where $X_S$ denotes the random variable
$$ X_S \coloneqq \sum_{j \in S} \lambda_{j} - \lambda_{j,n-1}.$$

It is convenient to exclude a small exceptional set to keep the eigenvalues $\lambda_j$ somewhat under control.  From Lemma \ref{lem:TV} we already know that there is a constant $C_0$ such that
\begin{equation}\label{la-1}
 |\la_j - \sigma_\lambda \gamma_j n^{1/2}| \leq n^{1/3} \min(j, n - j + 1)^{-1/3} \log^{C_0} n
\end{equation}
for all $1 \leq j \leq n$ with overwhelming probability.  From the Wegner estimate (see \cite[Theorem 3.5]{ESY}) and enlarging $C_0$ if needed, we also see that
\begin{equation}\label{la-2}
|\la_{j+1} - \lambda_j| \geq \exp(-\log^{C_0} n)
\end{equation}
with overwhelming probability for all $\eps n \leq j \leq (1-\eps) n$.  Thus, if we let $E$ denote the event that \eqref{la-1}, \eqref{la-2} both hold for all $\eps n \leq j \leq (1-\eps) n$, then $E$ holds with overwhelming probability; for future reference we also note the constraints \eqref{la-1}, \eqref{la-2} defining $E$ are restricting $\lambda$ to a certain convex subset of $\Spec$.  It suffices to show that
$$ \Var(X_S| E) = O(n^{2-c+o(1)}).$$
We split this by conditioning on the spectrum $\lambda$ of $A$.  
By the law of total variance (noting that the event $E$ is measurable with respect to $\lambda$), it suffices to establish the bounds
\begin{equation}\label{first-var}
\Var(\E( X_S | \lambda) | E) = O(n^{2-c+o(1)})
\end{equation}
and
\begin{equation}\label{second-var}
\E( \Var( X_S | \lambda) | E) = O(n^{2-c+o(1)}).
\end{equation}

To prove \eqref{second-var}, we expand out the left-hand side as
$$ \sum_{i,j \in S} \E( \cov( \lambda_{i} - \lambda_{i,n-1}, \lambda_{j} - \lambda_{j,n-1}|\lambda) | E)$$
where we use $\cov(X,Y) \coloneqq \E(XY) - (\E X)(\E Y)$ to denote the covariance between two random variables $X,Y$.
In the next section we will establish the bound
\begin{equation}\label{isep}
\E( \cov( \lambda_{i} - \lambda_{i,n-1}, \lambda_{j} - \lambda_{j,n-1}|\lambda) | E) \ll \frac{n^{o(1)}}{(1+|i-j|)^2}
\end{equation}
for all $2 \eps n \leq i,j \leq (1-2\eps) n$, which certainly implies \eqref{second-var}.  

In the remainder of this section we will reduce \eqref{first-var} to an estimate somewhat similar to \eqref{isep}, and then we will prove both remaining required inequalities in the next section.

Since the event $E$ is restricting $\lambda$ to a convex set in $\R^d$, so the probability distribution function of $\lambda$ is still log-concave after conditioning to $E$.  Thus Poincar\'e estimates such as Proposition \ref{log-concave-poin} become available.  As it turns out, a direct application of this proposition gives unfavorable estimates, basically because of long-range correlations between $\lambda_i$ and $\lambda_j$ make the operator norm of the inertia matrix large, and also because the known correlation decay estimates are currently only available in the bulk.  To resolve this, we do not use the standard basis $e_1,\dots,e_d$ of $\R^d$, but instead the following basis consisting of three groups:
\begin{itemize}
\item The vector $e_1+\dots+e_d$.
\item The vectors $e_{i+1}-e_i$ for $i$ in the bulk region $\mathtt{bulk} \coloneqq \{ i: \eps n \leq i < (1-\eps) n \}$.
\item The vectors $e_{i+1}-e_i$ for $i$ in the edge region $\mathtt{edge} \coloneqq \{ i: 1 \leq i < \eps n \hbox{ or } (1-\eps) n \leq i < n \}$. 
\end{itemize}
The point is that $\E(X_S|\lambda)$ has different behavior in each of the three groups of vectors.  In the direction $e_1+\dots+e_d$, the function $\E(X_S|\lambda)$ is in fact constant.  This is because once one conditions on $\lambda$, the random variable $\lambda_{j,n-1}$ has the distribution of the $j^{\mathrm{th}}$ largest eigenvalue of the top left $n-1 \times n-1$ minor of a Hermitian matrix chosen uniformly at random amongst all matrices with eigenvalue $\lambda$.  Moving $\lambda$ in the direction $e_1+\dots+e_d$ then amounts to shifting $\lambda_j$ and $\lambda_{j,n-1}$ by the same constant, so the expectation $\E(X_S|\lambda)$ remains unchanged.

As it turns out, $\E(X_S|\lambda)$ is significantly more sensitive to the bulk eigenvalue gaps $\lambda_{i+1}-\lambda_i$ than the edge eigenvalue gaps $\lambda_{j+1}-\lambda_j$.  To exploit this, we apply Proposition \ref{log-concave-poin} with suitable choices of weights (sending the weight on the basis vector $e_1+\dots+e_d$ to infinity) to conclude that
\begin{equation}\label{again}
\begin{split}
\Var(\E( X_S | \lambda) | E) &\ll \E\left( |\nabla_{\mathtt{bulk}} \E( X_S | \lambda) |^2 + n |\nabla_{\mathtt{edge}} \E( X_S | \lambda) |^2 | E\right) \\
&\quad \times \left(\|M_{\mathtt{bulk}}\|_{\mathrm{op}} + n^{-1} \|M_{\mathtt{edge}}\|_{\mathrm{op}}\right) \log n
\end{split}
\end{equation}
where for $\Omega = \mathtt{bulk}, \mathtt{edge}$ one has
$$
 |\nabla_{\Omega} \E( X_S | \lambda) |^2 \coloneqq \sum_{i \in \Omega} |(\partial_{\lambda_{i+1}} - \partial_{\lambda_i}) \E( X_S | \lambda) |^2 $$
and $M_{\Omega}$ is the covariance matrix with entries
$$ \cov( \lambda_{i+1} - \lambda_i, \lambda_{j+1} - \lambda_j | E )$$
for $i,j \in \Omega$.

We now estimate the entries of the covariance matrices $M_{\mathtt{bulk}}$, $M_{\mathtt{edge}}$.  For the edge region, we use \eqref{la-1} to conclude that $\lambda_{i+1} - \lambda_i = O( n^{1/3+o(1)})$ for all $i \in \mathtt{edge}$, hence all entries of $M_{\mathtt{edge}}$ have size $O( n^{2/3+o(1)})$.  By Schur's test, we conclude that
$$ \| M_{\mathtt{edge}}\|_{\mathrm{op}} \ll n^{5/3+o(1)}.$$
In the bulk region $\mathtt{bulk}$, the same argument gives the bound $\lambda_{i+1} - \lambda_i = O( n^{o(1)})$, hence all entries of $M_{\mathtt{bulk}}$ have size $O(n^{o(1)})$.  Schur's test then gives the bound
$$ \| M_{\mathtt{bulk}}\|_{\mathrm{op}} \ll n^{1+o(1)},$$
but this is not quite strong enough for our application.  To do better, we appeal to the results of Cipolloni, Erd\H{o}s and Schr\"{o}der \cite[Proposition 3.3, Case 1]{Cip}, which in our notation gives the covariance bound
$$ \cov( P_1( \lambda_{i+1} - \lambda_i) , P_2( \lambda_{j+1} - \lambda_j ) ) \ll_\eps n^{-\zeta_2} \|P_1 \|_{C^5} \|P_2 \|_{C^5}$$
whenever $i,j \in \mathtt{bulk}$, $|i-j| \geq n^{1-\zeta_1}$, and $P_1, P_2$ are bounded smooth compactly supported test functions, for some absolute constants $\zeta_1, \zeta_2 > 0$.  One can easily restrict this to the overwhelmingly probable event $E$ and conclude that
$$ \cov( P_1( \lambda_{i+1} - \lambda_i) , P_2( \lambda_{j+1} - \lambda_j ) | E ) \ll_\eps n^{-\zeta_2} \|P_1 \|_{C^5} \|P_2 \|_{C^5}.$$
By a suitable partition of unity, we then have the improved bound
$$ \cov( \lambda_{i+1} - \lambda_i, \lambda_{j+1} - \lambda_j | E ) \ll_\eps n^{-\zeta_2+o(1)}$$
on the entries of $M_{\mathtt{bulk}}$ far from the diagonal.  Applying Schur's test again, we now get the improved bound
$$ \| M_{\mathtt{bulk}}\|_{\mathrm{op}} \ll n^{1-\min(\zeta_1,\zeta_2)+o(1)},$$
and hence the quantity \eqref{again} may be bounded by
$$\ll n^{1 - \min(\zeta_1,\zeta_2,2/3)+o(1)} \E\left( |\nabla_{\mathtt{bulk}} \E( X_S | \lambda) |^2 + n |\nabla_{\mathtt{edge}} \E( X_S | \lambda) |^2 | E\right).$$
Thus, to prove \eqref{second-var}, it suffices to establish the bound
\begin{equation}\label{neat}
\E\left( |\nabla_{\mathtt{bulk}} \E( X_S | \lambda) |^2 + n |\nabla_{\mathtt{edge}} \E( X_S | \lambda) |^2 | E\right)
\ll n^{1+o(1)}.
\end{equation}
In the next section we will establish the bound
\begin{equation}\label{e-bound}
 \E (|\partial_{\lambda_i} \E( \lambda_j - \lambda_{j,n-1} |\lambda )|^2 | E ) \ll
n^{o(1)} (1 + n |\gamma_i - \gamma_j|)^{-4}
\end{equation} 
whenever $1 \leq i \leq n$ and $j \in \mathtt{bulk}$.  Taking square roots and summing over $j \in S$ using the triangle inequality, one obtains
$$ \E (|\partial_{\lambda_i} \E( X_S |\lambda )|^2 | E ) \ll n^{-2+o(1)}$$
for $i \in \mathtt{edge}$, and
$$ \E (|\partial_{\lambda_i} \E( X_S |\lambda )|^2 | E ) \ll n^{o(1)}$$
for $i \in \mathtt{bulk}$.  Summing in $i$, one obtains \eqref{neat} and thus \eqref{second-var}.

It thus remains to establish the bounds \eqref{isep}, \eqref{e-bound}.  This is the objective of the remaining sections of the paper.

\section{Determinantal process calculations}

In this section we fix $\la$ to be a deterministic element of $\Spec^\circ$, and let $A$ be a Hermitian matrix drawn uniformly at random amongst all matrices with eigenvalues $\la$.  We then let $x_1 \geq \dots \geq x_{n-1}$ be the eigenvalues of the top left $n-1 \times n-1$ minor of $A$.  In order to establish \eqref{isep}, \eqref{e-bound}, we would like to understand the mean and covariances of the gaps $\lambda_j - x_j$, as these random variables have the same distribution as $\lambda_j - \lambda_{j,n-1}$ conditioned to this choice of $\lambda$.  As it turns out, the theory of determinantal processes provide an explicit formula for these quantities:

\begin{proposition}[First and second moments]\label{moment-formulae}  With the above hypotheses, one has
\begin{equation}\label{first-moment}
 \E(\lambda_i - x_i) = \int_{I_i} Q_i(x)\ dx
\end{equation}
for all $1 \leq i \leq n$ and
\begin{equation}\label{second-moment}
\cov(\lambda_i - x_i, \lambda_j - \lambda_j) = \left( \int_{I_i} (1-Q_j(x))\ dx\right) \left( \int_{I_j} Q_i(x)\ dx \right)
\end{equation}
for all $1 \leq i < j < n$, where $I_j = I_{j,\lambda}$ is the interval $I_j \coloneqq [\lambda_{j+1},\lambda_j]$ and each $Q_j = Q_{j,\lambda}$ is the unique degree $n-1$ polynomial such that $Q_j(\la_i) = \one_{i\leq j}$ for $1 \leq i \leq n$.  More explicitly, by the Lagrange interpolation formula one has 
\beq Q_j(x) \coloneqq \sum_{i\leq j}\frac{\prod_{\ell \neq i}(x - \la_\ell)}{\prod_{\ell \neq i}(\la_i - \la_\ell)}.\eeq
\end{proposition}

\begin{proof}
By Proposition \ref{gt-rem}, each $x_i$ lies in $I_i$, with probability measure
\beq (n-1)! \frac{V_{n-1}(x)}{V_n(\la_1, \dots, \la_n)}\one_{I_1}(x_1)\dots \one_{I_{n-1}}(x_{n-1})dx_1\dots dx_{n-1}.\lab{eq:2.1}\eeq
As observed by Metcalfe \cite{metcalfe}, this law also has a determinantal form involving the polynomials $Q_j$ as follows.  From the fundamental theorem of calculus, the derivatives $Q'_j$ are degree $n-2$ polynomials that obey the mean zero conditions 
\beq \int_{I_i} Q_j'(x)dx = \one_{i=j}\lab{eq:2.3}\eeq
for $1 \leq i, j \leq n-1$, and thus form a basis of the polynomials of degree at most $n-2$. 
If one introduces the kernel $K \colon \R\times\R \ra \R$ by the formula
$$K(x, y) \coloneqq \sum_{j=1}^{n-1} \one_{I_j}(x)Q_j'(y)$$ then from \eqref{eq:2.3} we conclude that $K$ is a rank $n-1$ projection in the sense that $$\int_\R K(x, y)K(y, z) dy = K(x, z),$$ and $\int_\R K(y, y) dy = n-1$ for all $x, z \in \R$.
By the Gaudin lemma \cite{mehta}, we then have 
$$\int_{\R^{n-1}} \det(K(x_i, x_j))_{1 \leq i, j \leq n-1} dx_1 \dots dx_{n-1} = (n-1)!.$$ On the other hand this determinant is symmetric and supported on the $(n-1)!$ permutations of $I_1 \times \dots \times I_{n-1},$ hence 
$$\int_{I_1 \times \dots \times I_{n-1}} \det(K(x_i, x_j))_{1 \leq i, j \leq n-1}dx_1\dots dx_{n-1} = 1.$$
Because $Q_1', \dots, Q_{n-1}'$ is a basis of the polynomials of degree $n-2,$ we see that for $(x_1, \dots, x_{n-1})$ in $I_1 \times \dots \times I_{n-1},$ the determinant $$\det(K(x_i, x_j))_{1 \leq i, j\leq n-1} = \det(Q_i'(x_j))_{1 \leq i,j\leq n-1}$$ is a scalar multiple of the Vandermonde determinant, while also having a total mass $1$; comparing this with the probability measure (\ref{eq:2.1}), we see that this measure has the determinantal form 
$$\det(K(x_i, x_j))_{1 \leq i, j \leq n-1}\ dx_1\dots dx_{n-1}.$$ 
In particular (by another application of the Gaudin lemma) the one-point correlation function is $K(x, x)$ and the two point correlation function is $K(x, x) K(y, y) - K(x, y)K(y, x)$.  The identity \eqref{first-moment} then follows from integration by parts:
\beqs \E\la_i - x_i & = & \int_{I_i}(\la_i - x) K(x, x)\ dx\\
& = & \int_{I_i}(\la_i - x)Q_i'(x)\ dx\\
& = & \int_{I_i}Q_i(x)\ dx.\eeqs
A similar, but slightly lengthier computation gives \eqref{second-moment}:
 \beq \cov(\la_i - x_i, \la_j - x_j) & = & \int_{I_i}\int_{I_j} (\la_i - x)(\la_j - y)(K(x, x)K(y, y)- K(x, y) K(y, x))\ dx dy\nonumber\\
 &  & - \int_{I_i}(\la_i - x)K(x, x)dx\int_{I_j}(\la_j - y)K(y, y)\  dy\nonumber\\
 & = & - \int_{I_1}\int_{I_j}(\la_i - x) (\la_j - y) K(x, y) K(y, x)\ dx dy\nonumber\\
 & = & - \int_{I_i}\int_{I_j}(\la_i - x)(\la_j - y)Q_i'(y)Q_j'(x)\ dx dy\nonumber\\
 & = & \left(\int_{I_i}(1 - Q_j(y))dy\right)\left(\int_{I_j} Q_i(x)\ dx\right).\lab{eq:cov}\eeq
\end{proof}

To estimate $Q_j$ and its derivatives in $\lambda$, it will be convenient to use the following contour integration representation.

\begin{lemma}[Contour integral representation]\label{contour}  Let $P = P_\lambda$ denote the degree $n$ polynomial
$$P(x) \coloneqq\prod_{k=1}^n(x - \la_k).$$ 
Then for any $1 \leq j \leq n$ and $\sigma$ in the interior of $I_j$, one has
$$Q_j(x) = - \frac{1}{2\pi i} \int_{\sigma - i\infty}^{\sigma + i \infty} \frac{P(x)}{P(z)(x-z)}\ dz$$ 
for $x < \sigma$ and
$$1 - Q_j(x) = \frac{1}{2\pi i} \int_{\sigma - i\infty}^{\sigma + i\infty} \frac{P(x)}{P(z)(x-z)}\ dz$$ 
for $x > \sigma$.
\end{lemma}

\begin{proof} Observe that the rational function $Q_j(x)/P(x)$ decays at infinity and has poles at $\la_i,  i \leq j$ with residues $1/P'(\la_i)$,  thus we have the partial fractions decomposition
$$\frac{Q_j(x)}{P(x)} = \sum\limits_{i\leq j} \frac{1}{ P'(\la_i)(x-\la_i)}.$$
Similarly
$$\frac{1-Q_j(x)}{P(x)} = -\sum\limits_{i> j} \frac{1}{ P'(\la_i)(x-\la_i)}.$$
The claim now follows from the residue theorem.
\end{proof}

Now let $1 \leq i,j \leq n$.  Using the identity
$$ \frac{P(x)}{P(z)(x-z)} = \prod_{1 \leq k \leq n: k \neq i} \frac{x-\la_k}{z-\la_k} \left( \frac{1}{z-\la_i} + \frac{1}{x-z} \right)$$
we have
$$ \partial_{\la_i} \frac{P(x)}{P(z)(x-z)} = \frac{1}{(z-\la_i)^2} \prod_{1 \leq k \leq n: k \neq i} \frac{x-\la_k}{z-\la_k} $$
and so on differentiating under the integral sign we obtain
\begin{equation}\label{diff}
 \partial_{\la_i} Q_j(x) = - \frac{1}{2\pi i} \int_{\sigma - i\infty}^{\sigma + i \infty} \frac{1}{(z-\la_i)^2} \prod_{1 \leq k \leq n: k \neq i} \frac{x-\la_k}{z-\la_k}\ dz 
\end{equation}
whenever $\sigma$ is in the interior of $I_j$ and $x \neq \sigma$.  By continuity the restriction $x \neq \sigma$ can then be dropped.  Setting $x=\sigma$, which implies $|z-\la_k| \geq |x-\la_k|$, we conclude from the triangle inequality that
\begin{equation}\label{laq}
|\partial_{\la_i} Q_j(x)| \leq \frac{1}{2\pi} \int_{\R} \frac{1}{|x-\la_i+it|^2}\ dt = \frac{1}{2 |x - \la_i|}.
\end{equation}

\section{Conclusion of the argument}

We can now prove \eqref{isep} and \eqref{e-bound}.

We begin with \eqref{e-bound}.  Fix $1 \leq i \leq n$ and $j \in \mathtt{bulk}$.  By Proposition \ref{moment-formulae}, the left-hand side of \eqref{e-bound} is
$$ \E \left(|\partial_{\lambda_i} \int_{I_{j,\lambda}} Q_{j,\lambda}(x)\ dx|^2\big | E \right).$$
We divide the interval $I_{j,\lambda} = [\lambda_{j+1},\lambda_j]$ into the left half $I_{j,\lambda}^l = [\lambda_{j+1},\frac{\lambda_{j+1}+\lambda_j}{2}]$ and the right half $I_{j,\lambda}^r = [\frac{\lambda_{j+1}+\lambda_j}{2}, \lambda_j]$.  We shall just establish the bound
\begin{equation}\label{ej}
 \E \left(|\partial_{\lambda_i} \int_{I_{j,\lambda}^r} Q_{j,\lambda}(x)\ dx|^2\big | E \right) \ll
n^{o(1)} (1 + n |\gamma_i - \gamma_j|)^{-4};
\end{equation}
similar arguments apply for the left half $I_{j,\lambda}^l$, and the claim \eqref{e-bound} will then follow from the triangle inequality.

The quantity $\int_{I_{j,\lambda}^r} Q_{j,\lambda}(x)\ dx$ is unchanged if all of the $\lambda$ are shifted by the same constant.  In particular
$$ \sum_{i=1}^n \partial_{\lambda_i} \int_{I_{j,\lambda}^r} Q_{j,\lambda}(x)\ dx = 0.$$
Thus it will suffice to establish \eqref{ej} under the additional hypothesis $i \neq j$, as the excluded case $i=j$ is then handled by the triangle inequality.  The point of this reduction is that it generates a separation between $\lambda_i$ and $I_{j,\lambda}^r$.  (For the left half $I_{j,\lambda}^l$, one would instead enforce the hypothesis $i \neq j+1$.)

Henceforth $i \neq j$.  If $i$ is also not equal to $j+1$, we of course have
$$ \partial_{\lambda_i} \int_{I_{j,\lambda}^r} Q_{j,\lambda}(x)\ dx = \int_{I_{j,\lambda}^r} \partial_{\lambda_i} Q_{j,\lambda}(x)\ dx.$$
For $i = j+1$, we acquire an additional term of $\frac{1}{2} Q_{j,\lambda}(\frac{\la_{j+1}+\la_j}{2})$.  Thus, it will suffice to establish the bounds
\begin{equation}\label{ej1}
 \E \left(|\int_{I_{j,\lambda}^r} \partial_{\lambda_i} Q_{j,\lambda}(x)\ dx|^2 | E\right ) \ll
n^{o(1)} (1 + n |\gamma_i - \gamma_j|)^{-4}
\end{equation}
whenever $i \neq j$, as well as the additional bound
\begin{equation}\label{ej2}
 \E \left(|Q_{j,\lambda}\left(\frac{\la_{j+1}+\la_j}{2}\right)|^2\big | E \right) \ll n^{o(1)}.
\end{equation}

By \eqref{la-1}, $I_{j,\lambda}$ is contained in a fixed interval $I_j^*$ of length $n^{o(1)}$ centered around $\sigma_{\lambda} \sqrt{n} \gamma_j$, thus by Cauchy--Schwarz
$$ \left|\int_{I_{j,\lambda}^r} \partial_{\lambda_i} Q_{j,\lambda}(x)\ dx\right|^2 
\ll n^{o(1)} \int_{I_j^*} |\partial_{\lambda_i} Q_{j,\lambda}(x)|^2 \one_{x \in I^r_{j,\lambda}}\ dx.$$
By the triangle inequality, \eqref{ej1} will then follow from the pointwise bound
\begin{equation}\label{ej3}
  \E\left( |\partial_{\lambda_i} Q_{j,\lambda}(x)|^2 \one_{x \in I^r_{j,\lambda}}\big | E \right) \ll n^{o(1)} (1 + n |\gamma_i - \gamma_j|)^{-4}
	\end{equation}
for each $x \in I_j^*$.  

First consider the case where $|i-j| \geq \log^{2C_0} n$.  Applying \eqref{diff} with $\sigma=x$ and the triangle inequality, we have
$$
 \partial_{\la_i} Q_j(x) \ll \int_{\R} \frac{1}{(x-\la_i)^2} \prod_{1 \leq k \leq n: k \neq i} \frac{|x-\la_k|}{|x-\la_k+it|}\ dt.$$
From \eqref{la-1}, we can compute $|x-\la_i| = n^{o(1)} (1 + n |\gamma_i - \gamma_j|)$ and $\prod_{1 \leq k \leq n: k \neq i} \frac{|x-\la_k|}{|x-\la_k+it|} \ll \frac{n^{o(1)}}{1+t^2}$, and \eqref{ej3} follows in this case.

Now suppose $|i-j| < \log^{2C_0} n$, so that the right-hand side of \eqref{ej3} simplifies to $n^{o(1)}$.  By \eqref{laq}, we can bound 
$$ |\partial_{\lambda_i} Q_{j,\lambda}(x)|^2 \one_{x \in I^r_{j,\lambda}} \ll \frac{1}{\lambda_{j} - \lambda_{j+1}} + \frac{1}{\lambda_{j-1} - \lambda_j}$$
(by splitting into the cases $i<j$ and $i>j$).  Thus it will suffice to establish the bound
$$ \E\left(\frac{1}{(\lambda_j - \lambda_{j+1})^2}| E\right) \ll n^{o(1)}$$
(the claim for $\frac{1}{\lambda_{j-1} - \lambda_j}$ is of course similar).  Letting $K(x,y)$ be the determinantal kernel of the rescaled GUE matrix $A$ it suffices to show that
$$ \int_{I_j^*} \int_{I_j^*} \frac{K(x,x) K(y,y) - K(x,y) K(y,x)}{|x-y|^2}\ dx dy \ll n^{o(1)}.$$
But from the well known local smooth convergence of this kernel to a rescaled Dyson sine process (see e.g., \cite{mehta}), the integrand is $O(1)$, and the claim follows.  This proves \eqref{ej3}.

Finally, we need to show \eqref{ej2}.  Write $x = \frac{\la_{j+1}+\la_j}{2}$.  By Lemma \ref{contour} and the Plemelj formula, we can write
$$ Q_{j,\lambda}(x) = \frac{1}{2} - \frac{1}{2\pi} \mathrm{p.v.} \int_\R \frac{P(x)}{P(x+it)} \frac{dt}{t}.$$
Using \eqref{la-1}, \eqref{la-2} (which among other things makes $P(x)/P(x+it)$ very close to $1$ for $|t| \leq \exp(-\log^{2C_0} n)$, bounded in magnitude by $1$ for all $t$, and decaying fast for $|t| \geq n^{C_0}$ (say)), one can calculate that this integral is $O(n^{o(1)})$, giving \eqref{ej2}.  This completes the proof of \eqref{e-bound}.

Now we show \eqref{isep}.  Let $2\eps n \leq i,j \leq (1-2\eps) n$. If $|i-j| \leq \log^{2C_0} n$ then the claim follows from the crude bounds
$$ \lambda_i - \lambda_{i,n-1}, \lambda_j - \lambda_{j,n-1} = O(n^{o(1)})$$
from interlacing and \eqref{la-1}, so we may assume by symmetry that $j-i > \log^{2C_0} n$.  Applying \eqref{second-moment}, it suffices to show the pointwise bound
$$ \int_{I_j} \int_{I_i} (1-Q_j(x)) Q_i(y)\ dx dy \ll \frac{n^{o(1)}}{(j-i)^2}.$$
Applying Lemma \ref{contour}, we can write the left-hand side as
$$ \frac{1}{4\pi^2} \int_{I_j} \int_{I_i} \int_{x-i\infty}^{x+i\infty} \int_{y-i\infty}^{y+i\infty} \frac{P(x) P(y)}{P(w) P(z) (y-z)(x-w)}\ dw dz dx dy.$$
From the separation of $i,j$, we have the lower bounds
$$ |y-z|, |x-w| \gg j-i.$$
The quantity $|P(x)|/|P(z)|$ is bounded by $1$, and from \eqref{la-1} it is also bounded by $O(n^{o(1)}/|\mathrm{Im} z|^2)$ when $\mathrm{Im} z \geq \log^{2C_0} n$.  Similarly for $|P(y)|/|P(w)|$.  Also, from \eqref{la-1} the intervals $I_i, I_j$ have length $O(n^{o(1)})$, and the claim follows.

\section{Open questions}
\ben 

\item What can be said about the concentration of random real valued augmented hives with general boundary conditions? If they do concentrate,  what are the possible subsequential limit shapes? In particular, is the limit unique? In the limit when one of the boundary conditions is more spread out than the other, the limit shape should essentially degenerate to fractional free convolution powers \cite{dima}, in analogy with Proposition \ref{gt-rem}(iv).
\item Do the local statistics of the random augmented GUE hive process converge (either in the bulk or the edge) to a known limit?  In the case of the random Gelfand--Tsetlin process, the limit is known to essentially be the Boutillier bead process \cite{boutillier}; see \cite{metcalfe}.
\item Do random integer valued augmented hives with general boundary conditions concentrate? Again, if they do concentrate,  what are the possible subsequential limit shapes? In particular, is the limit unique?
\een

\appendix

\section{Verification of the octahedron recurrence}\label{octahedron-app}

In this appendix we review the relevant material from \cite{KT2} needed to establish Theorem \ref{octahedron}.  We will assume that the reader has some familiarity with the material in that paper.

The first step is to view the four triangles $T, T', U, U'$ as faces of the tetrahedron 
$$ \mathtt{tet} \coloneqq \{ [x,y,z,w] \in \Z^4: x,y,z,w \geq 0; x+y+z+w=n \}$$
(which has vertices $(n,0,0,0), (0,n,0,0), (0,0,n,0), (0,0,0,n)$) as follows:

\begin{itemize}
\item A point $(i,j) \in T$ can be identified with the point $(0, j-i, n-j, i)$, thus identifying $T$ with the triangle $[T]$ with vertices $(0,n,0,0)$, $(0,0,n,0)$, $(0,0,0,n)$.
\item A point $(i,j) \in T'$ can be identified with the point $(i-j, 0, n-i, j)$, thus identifying $T'$ with the triangle $[T']$ with vertices $(n,0,0,0)$, $(0,0,n,0)$, $(0,0,0,n)$.
\item A point $(i,j) \in U$ can be identified with the point $(n-j,n-i,0,i+j-n)$, thus identifying $U$ with the triangle $[U]$ with vertices $(n,0,0,0)$, $(0,n,0,0)$, $(0,0,0,n)$.
\item A point $(i,j) \in U'$ can be identified with the point $(i, j, n-i-j,0)$, thus identifying $U$ with the triangle $[U']$ with vertices $(n,0,0,0)$, $(0,n,0,0)$, $(0,0,0,n)$.
\end{itemize}
Note that these identifications are consistent along the shared edge $T \cap T'$ of $T$ and $T'$, the shared edge $U \cap U'$ of $U$ and $U'$, and on the boundary of $\{0,\dots,n\}^2$.  The function $\tilde h \colon T \cup T' \to \R$ can then, by abuse of notation, be thought of as a function from the upper faces $[T] \cup [T']$ of $\mathtt{tet}$ to $\R$, while the function $\tilde k \colon U \cup U' \to \R$ can similarly be thought of as a function from the lower faces $[U] \cup [U']$ of $\mathtt{tet}$ to $\R$, using the above identifications.

As discussed in \cite[\S 5]{KT2}, the tetrahedron $\mathtt{tet}$ decomposes into a certain number of unit tetrahedra (in two different orientations), as well as some unit octahedra with vertices
$$ [x+1,y+1,z,w], [x+1,y,z+1,w], [x+1,y,z,w+1], [x,y+1,z+1,w], [x,y+1,z,w+1], [x,y,z+1,w+1].$$
We can place a partial ordering on $\mathtt{tet}$ by declaring $[x,y,z,w] \prec [x',y',z',w']$ if $x \geq x'$, $y \geq y'$, $z \leq z'$, and $w \leq w'$, so that $[x+1,y+1,z,w]$ as the minimal vertex of this octahedron and $[x,y,z+1,w+1]$ is the maximal vertex.  The \emph{octahedron recurrence} is then a relation between $\tilde h \colon [T] \cup [T'] \to \R$ and $\tilde k \colon [T] \cup [T'] \to \R$, namely that there exists a common extension $o \colon \mathtt{tet} \to \R$ of $\tilde h, \tilde k$ to the entire simplex $\mathtt{tet}$ that obeys the \emph{(tropical) octahedron rule}
$$ o( [x,y,z+1,w+1] ) = \max( o([x+1,y,z+1,w]) + o([x,y+1,z,w+1]),$$ $$ o([x+1,y,z,w+1]) + o([x,y+1,z+1,w])) - o([x+1,y+1,z,w])$$
for every octahedron in $\mathtt{tet}$.  As established in \cite[\S 5]{KT2}, this rule uniquely determines $\tilde h$ as a function of $\tilde k$, and is a bijection between pairs of hives $(k,k')$ and pairs of hives $(h,h')$ with compatible boundary data as indicated in Section \ref{recurrence-sec}.

Let $v = (i,j) \in T \cup T'$, and let $b \in [T] \cup [T']$ be the corresponding point in the lower faces of $\mathtt{tt}$.  
As discussed in \cite[\S 7]{KT2}, to evaluate $\tilde h$ at $b$, one must ``excavate'' all of the tetrahedra and octahedra in the ``light cone''
$$ \{ c \in \mathtt{tet}: c \succ b \}$$
and in particular one needs to evaluate $\tilde k$ at all points $(i',j')$ whose corresponding element in $[U] \cup [U']$ lies above $b$ in the partial ordering $\prec$ on $\mathtt{tet}$.  This requirement can be stated more explicitly as follows:
\begin{itemize}
\item  If $(i,j) \in T$ and $(i',j') \in U$, then we require $(0, j-i, n-j, i) \prec (n-j',n-i',0,i'+j'-n)$.
\item  If $(i,j) \in T$ and $(i',j') \in U'$, then we require $(i-j, 0, n-i, j) \prec (i', j', n-i'-j',0)$.
\item  If $(i,j) \in T'$ and $(i',j') \in U$, then we require $(0, j-i, n-j, i) \prec (n-j',n-i',0,i'+j'-n)$.
\item  If $(i,j) \in T'$ and $(i',j') \in U'$, then we require $(i-j, 0, n-i, j) \prec (i', j', n-i'-j',0)$.
\end{itemize}
A tedious but routine calculation then shows that these constraints are equivalent to $(i',j')$ lying in the hexagon $\hexagon_v$ defined in
Definition \ref{octa-def}.

In \cite[\S 7]{KT2}, a dual graph $G_b$ to this hexagon $\hexagon_v$ is then formed.  The vertices of this graph correspond the unit triangles in $\hexagon_v$, except that the upper and lower boundary triangles associated to a given boundary edge have been identified into a single vertex.  Two vertices in this graph are adjacent if they can correspond to unit triangles that share a common edge, or equivalently if there is a lozenge comprising of two unit triangles associated to the indicated vertices.  Thus, there is a one-to-one correspondence between edges in $G_b$ and lozenges in $\hexagon_v$.  Interior vertices of $\hexagon_v$ correspond to interior faces of $G_b$ (which are a rhombus for vertices on the equator, and a hexagon otherwise), while boundary vertices of $\hexagon_v$ can be identified with ``external faces'' of $G_b$, of which the two equatorial ones are ``external rhombi'' and the remainder ``external hexagons''; see Figure \ref{fig:dual}.

\begin{figure}
\begin{center}
\includegraphics[scale=0.40]{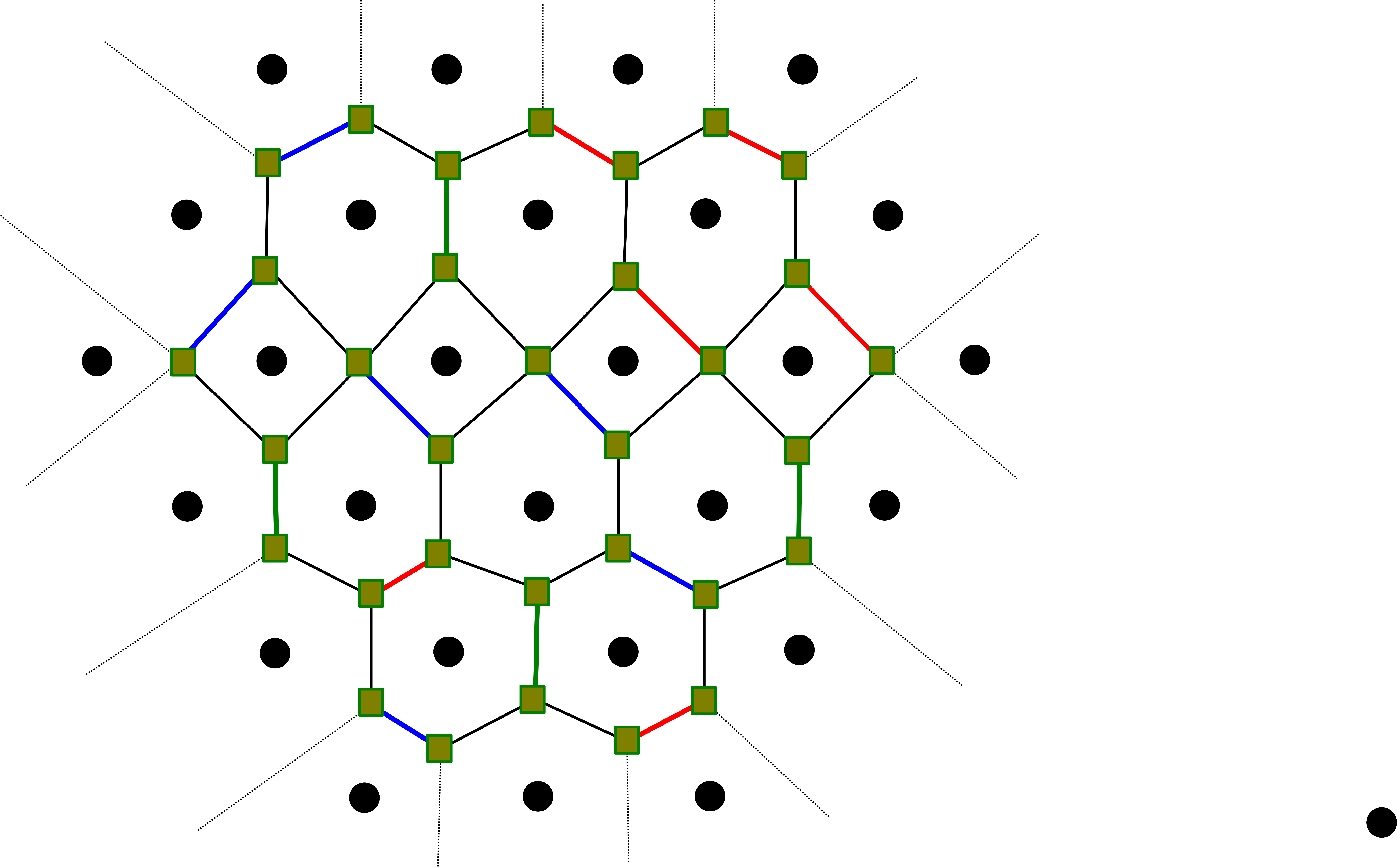}
\caption{The dual graph $G_b$ of the hexagon $\hexagon_v$ appearing in Figures \ref{fig:standard}, \ref{fig:typical}, where the vertices are depicted as olive squares; the vertices of $\hexagon_v$ then are associated to interior or external faces of $G_b$ as indicated.  The colored edges represent a matching of $G_b$, associated to the lozenge tiling in Figure \ref{fig:typical}. Compare with \cite[Figures 6, 7]{KT2}.}\label{fig:dual}
\end{center}
\end{figure}

Following the work of Speyer \cite{Speyer}, one then considers perfect matchings $\mu$ on the graph $G_b$, that is to say collections of edges where each vertex is incident to precisely one edge in the matching.  By the above discussion, each edge in the matching corresponds to a lozenge, and every non-boundary triangle will be covered by precisely one lozenge, with each boundary edge being incident to precisely one boundary triangle covered by a lozenge.  In other words, there is a one-to-one correspondence between perfect matchings on the graph $G_b$ and lozenge tilings of $\hexagon_v$. Again, we refer the reader to Figure \ref{fig:dual} for an example of this correspondence.

To each perfect matching $\mu$ on $G_b$, one can associate a linear form (or ``monomial'', in non-tropical notation)
$$ m_\mu = \sum_{w \in \hexagon_v} c_w \tilde k(w),$$
where for each $w \in \hexagon_v$ (corresponding to an interior rhombus, interior hexagon, or external face in $G_b$), the coefficient $c_w \in \Z$ is defined as follows:
\begin{itemize}
\item If $w$ is an interior hexagon, then $c_w$ is two minus the number of edges in $\mu$ adjacent to $w$.
\item If $w$ is an interior rhombus or external hexagon, then $c_w$ is one minus the number of edges in $\mu$ adjacent to $w$.
\item If $w$ is an exterior rhombus, then $c_w=0$.
\end{itemize}

The main theorem of Speyer \cite{Speyer}, reproduced in \cite[\S 7]{KT2} in the case of excavating a tetrahedron and using ``tropical'' notation, then asserts that the result $\tilde h(v)$ of the octahedron recurrence applied at $v$ is equal to the maximum of the $m_\mu$ over all choices of perfect matchings $\mu$.  To obtain Theorem \ref{octahedron}, it remains to show that the linear form $m_\mu$ defined here agrees with the weight $w_\Xi$ of the associated lozenge tiling $\Xi$ defined in \eqref{weight-form}.  It suffices to show that, for each $w \in \hexagon_v$, the coefficients of $\tilde k(w)$ for $m_\mu$ and $w_\Xi$ match.  This is accomplished as follows:

\begin{itemize}
\item If $w$ is an non-equatorial interior point of $\hexagon_v$ (thus generating an interior hexagon), then the coefficient of $w_\Xi$ is given by $\frac{1}{3} a - \frac{1}{3} o$, where $a, o$ are the number of acute and obtuse angles subtended by the components of the lozenge tiling at $w$, while the coefficient of $m_\mu$ is $2-o$.  Since the total angle around $w$ is $2\pi$, we have $a\frac{\pi}{3} + o \frac{2\pi}{3} = 2\pi$, hence $\frac{1}{3} a - \frac{1}{3} o = 2-o$ as required.
\item If $w$ is an equatorial interior point of $\hexagon_v$ (thus generating an interior rhombus), then the coefficient of $w_\Xi$ is $\frac{1}{3} (a-2) + 0 \times 1 - \frac{1}{3} (o+1)$ (since of the two boundary triangles in the tiling adjacent to $w$ and subtending an acute angle, one has a coefficient of $0$ and the other has a coefficient of $-\frac{1}{3}$), while the $m_\mu$ coefficient is $1-o$.  Again we have $a\frac{\pi}{3} + o \frac{2\pi}{3} = 2\pi$, giving $\frac{1}{3} (a-2) + 0 \times 1 - \frac{1}{3} (o+1) = 1-o$ as required.
\item If $w$ is a boundary point of $\hexagon_v$ that is not a vertex $A, B, C, D, E, F$, then the coefficient of $w_\Xi$ is $\frac{1}{3} a - \frac{1}{3} o$ and the coefficient of $m_\mu$ is $1-o$.  In this case the total angle $a\frac{\pi}{3} + o \frac{2\pi}{3}$ is equal to $\pi$, so 
we have $\frac{1}{3} a - \frac{1}{3} o = 1-o$ as required.
\item If $w$ is one of the vertices $B, D, E, F$, then the coefficient of $w_\Xi$ is $\frac{1}{3} (a+1) - \frac{1}{3} o$ and the coefficient of $m_\mu$ is $1-o$.  The total angle $a\frac{\pi}{3} + o \frac{2\pi}{3}$ is equal to $\frac{2\pi}{3}$, so we have $\frac{1}{3} (a+1) - \frac{1}{3} o = 1-o$ as required.
\item If $w$ is equal to $A$ or $C$, one easily checks that the coefficient of either $w_\Xi$ or $m_\mu$ vanish.
\end{itemize}

This establishes Theorem \ref{octahedron}.

\bibliographystyle{alpha}

\end{document}